\documentclass[openany, 12pt]{amsart} 

\usepackage{geometry} 
\usepackage{etoolbox}

\patchcmd{\tocchapter}{#1}{\MakeUppercase{#1}}{}{}

\makeatletter
\patchcmd{\@tocline}
  {\hfil}
  {\leaders\hbox{\,.\,}\hfil}
  {}{}

\newcommand{\R}{\mathbb{R}}  
\newcommand{\Z}{\mathbb{Z}}
\newcommand{\N}{\mathbb{N}}
\newcommand{\Q}{\mathbb{Q}}
\newcommand{\C}{\mathbb{C}}

\newcommand{\F}{\mathbb{F}}

\usepackage{mathrsfs}
\usepackage{hyperref}
\hypersetup{
    colorlinks=false,
    linkcolor=black,
    filecolor=blue,      
    urlcolor=blue,
    citecolor=black
    }
\DeclareFontFamily{U}{mathx}{}
\DeclareFontShape{U}{mathx}{m}{n}{<-> mathx10}{}
\DeclareSymbolFont{mathx}{U}{mathx}{m}{n}
\DeclareMathAccent{\widehat}{0}{mathx}{"70}
\DeclareMathAccent{\widecheck}{0}{mathx}{"71}
\usepackage[table, dvipsnames]{xcolor}
\usepackage{booktabs} 
\usepackage{array} 
\usepackage{paralist} 
\usepackage{verbatim} 
\usepackage{graphicx} 
\usepackage{tikz-cd}
\usepackage{siunitx}
\usepackage{amsfonts}
\usepackage{amssymb}
\usepackage{enumitem}
\usepackage{moreenum}
\usepackage{mathtools}
\usepackage{upgreek}
\usepackage{physics}
\usepackage{amssymb,amscd,amsmath,amsthm,url}
\usepackage{thmtools}
\usepackage{caption}
\usepackage{subcaption}
\usepackage{esint}
\usepackage{sseq}
\usepackage[nameinlink]{cleveref}
\usepackage{needspace}

\numberwithin{equation}{section}

\usepackage{cancel}

\newcommand{\pd}{\partial}
\newcommand{\can}{\mathrm{can}}
\newcommand{\pre}{\mathrm{pre}}
\newcommand{\w}{\omega}
\newcommand{\std}{\mathrm{std}}

\newcommand{\SL}{\textup{SL}}

\newcommand{\im}{\operatorname{im}}

\setlist[itemize]{leftmargin=*}

\renewcommand{\tilde}{\widetilde}
\renewcommand{\hat}{\widehat}
\usepackage{stmaryrd}

\newcommand{\ECH}{\mathrm{ECH}}

\numberwithin{equation}{section}

\newtheorem{theorem}{Theorem}[section]
\newtheorem{proposition}[theorem]{Proposition}
\newtheorem{corollary}[theorem]{Corollary}
\newtheorem{lemma}[theorem]{Lemma}

\theoremstyle{definition}
\newtheorem{definition}[theorem]{Definition}
\newtheorem{remark}[theorem]{Remark}
\newtheorem{example}[theorem]{Example}
\newtheorem{question}[theorem]{Question}
\newtheorem{assumption}[theorem]{Assumption}

\usepackage{adjustbox}

\usepackage{csquotes}

\definecolor{mypink2}{RGB}{255, 186, 239}
\definecolor{myyellow}{RGB}{255, 227, 89}
\definecolor{mygreen}{RGB}{178, 252, 141}

\makeatletter
\def\l@subsection{\@tocline{2}{0pt}{2.5pc}{5pc}{}}
\makeatother

\begin{document}
\title{Infinite ECH Capacities and Anosov Flows}

\author{Gabriel Beiner}
\address{Gabriel Beiner, 935 Evans Hall, University of California, Berkeley}
\email{gabriel.beiner@berkeley.edu}
\begin{abstract}
	This article relates the theory of embedded contact homology (ECH) with the dynamics of Anosov flows. We show that in many cases the ECH capacities of a symplectic 4-manifold are infinite, including cotangent disk bundles over closed oriented surfaces of genus at least two. We prove that ECH obstructs Reeb Anosov and Hamiltonian Anosov flows, addressing the four-dimensional case of a question posed by Herman in 1998. Further, we obtain Floer-theoretic obstructions to a 3-manifold admitting any Anosov flow. As an application, we give new constraints on the existence of embedded Lagrangians of genus at least two in symplectic 4-manifolds.  In an appendix, some related results in all dimensions are proved for capacities constructed from rational symplectic field theory.
\end{abstract}
\maketitle

{ \setcounter{tocdepth}{2}
\tableofcontents  }

\section{Introduction}
Given a four-dimensional Liouville domain $(X,\w)$, Hutchings \cite{HutQ} introduced a sequence of associated symplectic invariants, called the \emph{ECH capacities}:
\[0 < c_1^{\ECH}(X,\w) \leq c_2^{\ECH}(X,\w)\leq c_3^{\ECH}(X,\w)\leq \cdots \leq +\infty.\]
These capacities are defined using spectral invariants obtained from the embedded contact homology (ECH) of the boundary of $X$. That is, we associate to a non-degenerate contact 3-manifold $(Y,\lambda)$ its \emph{ECH spectrum}:
\[0 < c_1(Y,\lambda) \leq c_2(Y,\lambda)\leq c_3(Y,\lambda) \leq \cdots \leq +\infty,\]
given as the minimal action needed to represent certain classes in $ECH(Y,\lambda)$; the ECH capacities of $(X,\w)$ are defined to agree with the ECH spectrum of the contact-type boundary of $X$. They are ``capacities" in the sense that each invariant  $c_k^{\mathrm{ECH}}$ scales linearly with rescaling the symplectic form, and if there exitsts a symplectic embedding $\varphi: (X,\w) \hookrightarrow (X',\w')$, then $c^\ECH_k(X,\w)\leq c_k^{\ECH}(X',\w')$ for all $k$.

The ECH capacities give many powerful insights into four-dimensional symplectic embedding problems \cite{CCGFHR, CG}. The capacities (when finite), and a more general family of ECH spectral invariants, are also known to satisfy a ``Weyl law" \cite{Weyl} which is instrumental in the proofs of groundbreaking results on the dynamics of three-dimensional Reeb flows (e.g. \cite{CGH, Iri, CGP}).  

This paper concerns some relations between ECH capacities and the theory of Anosov flows. More precisely, we will show that finiteness of the ECH capacities obstructs the existence of Anosov flows in various settings. This will in turn allow us to determine that the ECH capacities are infinite in many new cases.

\subsection{Statement of main results}\label{ssec:mainresult}
In order to study symplectic embeddings and dynamics in dimension four, it is an important problem to compute the ECH capacities for a given Liouville domain. Examples where the capacities have been determined include convex and concave toric domains \cite{HutQ, CCGFHR}, and the unit cotangent disk bundles of $S^2$ \cite{FR, FRV}, $\R P^2$ \cite{FR}, $T^2$ \cite{HutQ}, and the Klein bottle \cite{MR} with respect to standard metrics. In general, it may be difficult to even determine if the ECH capacities are finite---this has been an open question for the cotangent disk bundles of all higher genus surfaces. 
Our first main result concerns a simple condition under which all ECH capacities must be infinite; in particular, it implies the capacities are infinite for all cotangent disk bundles over oriented surfaces of genus at least two. We begin by recalling some standard terminology.

\begin{definition}
Consider closed (possibly disconnected) contact manifolds $(Y_+, \lambda_+)$ and $(Y_-,\lambda_-)$ and a symplectic manifold $(X,\w)$ with oriented boundary $\pd X = Y_+\sqcup -Y_-$.

\begin{itemize}
	\item $(X,\w)$ is a \emph{strong symplectic cobordism} from $(Y_+,\lambda_+)$ to $(Y_-,\lambda_-)$ if $\w$ agrees with $d\lambda_{\pm}$ on $Y_{\pm}$.
	\item $(X,\w)$ is a \emph{weakly exact symplectic cobordism} from $(Y_+,\lambda_+)$ to $(Y_-,\lambda_-)$ if it is a strong symplectic cobordism and $\w$ is exact.
	\item $(X,\w)$ is an \emph{exact symplectic cobordism}  if $\w$ has a global primitive $\lambda$, called a \emph{Liouville form}, which restricts to $\lambda_\pm$ on $Y_\pm$.
	 \item  $(X,\w)$ is a \emph{weak symplectic cobordism} from $(Y_+, \xi_+)$ to $(Y_-,\xi_-)$ if $\w|_{\xi_{\pm}}> 0$. 
	 \end{itemize}
In each case, we call $Y_+$ (respectively $Y_-$) the \emph{convex} (resp.\!\! \emph{concave}) boundary of $X$. When a cobordism of one of the above types has empty concave boundary, it is called a \emph{filling}. A (weakly) exact filling is also called a \emph{(weak) Liouville domain}. When a strong cobordism has empty convex boundary, we call it a \emph{strong symplectic cap}. When we speak of strong or  exact cobordisms for  $(Y_\pm,\xi_\pm)$, we mean there are contact forms $\lambda_\pm$ for $\xi_\pm$ for which the corresponding cobordism is strong or  exact.
\end{definition}

\subsubsection{Capacities of domains with disconnected boundary}
Our first theorem concerns the ECH capacities of a Liouville domain with disconnected boundary.

\begin{theorem}
\label{thm:inf_spec_invts}
Suppose $(X,\w)$ is a 4-dimensional connected weak Liouville domain with disconnected boundary. Then the ECH spectrum of each boundary component and the ECH capacities of $(X,\w)$  are infinite.
\end{theorem}

We note that \Cref{thm:inf_spec_invts} will follow purely from  formal properties of ECH. Nevertheless, it can be applied to yield a number of results exposing connections between Anosov dynamics and symplectic and contact geometry.

\subsubsection{Anosov Liouville domains}
The main examples of Liouville domains with disconnected boundary have the following form. We say that a closed $3$-manifold $Y$ admits a \emph{Liouville pair} $(\alpha_+, \alpha_-)$ if $\alpha_+$ and $\alpha_-$ are respectively positive and negative contact forms on $Y$ and 
\[\lambda = e^{-t}\alpha_- + e^{t}\alpha_+\]
is a Liouville form on $\R_t\times Y$. For $N$ large enough, $[-N,N]\times Y$ is a Liouville filling of $(Y,  \ker \alpha_+) \sqcup (-Y,  \ker \alpha_-)$. \Cref{thm:inf_spec_invts} implies the ECH spectra of the contact structures defined by either element of a Liouville pair are infinite.

Many examples of three-dimensional Liouville pairs come from the theory of Anosov flows. A closed oriented 3-manifold $Y$ admits an \emph{Anosov flow} $\psi^t$ if there is a pair of one-dimensional ``stable" and ``unstable" foliations on $Y$ transverse to the flow whose tangent vectors exponentially contract and expand under $\psi^t$; a precise definition is given in \Cref{ssec:Ano}. We will always consider our flows to be smooth. We will also almost always assume our Anosov flows are oriented.
\begin{definition}
	We say an Anosov flow on $Y$ is \emph{oriented Anosov} if its stable and unstable foliations are orientable and given orientations compatible with that of $Y$.
\end{definition}
Given an oriented Anosov flow on $Y$, Eliashberg--Thurston \cite{ET} and Mitsumatsu \cite{Mit} show that $Y$ admits a pair of contact structures $(\xi_+ = \ker \alpha_+, \xi_- = \ker \alpha_-)$ whose transverse intersection $\xi_+ \pitchfork\xi_-$ spans the flow, and so that $(\alpha_+,\alpha_-)$ is a  Liouville pair. Following \cite{CLMM}, we call the corresponding domain $I\times Y$, which gives an exact filling of $(Y,\xi_+)\sqcup (-Y,\xi_-)$, an \emph{Anosov Liouville domain}. 

While the general construction of Anosov Liouville domains is due to Mitsumatsu \cite{Mit},  some explicit examples of these domains were described earlier by McDuff \cite{McD} and Geiges \cite{Gei}. The details of these constructions are given in \Cref{sec:defs}. For any contact structure on the boundary of an Anosov Liouville domain, \Cref{thm:inf_spec_invts} will imply its ECH spectrum is infinite (see \Cref{cor:inf_spec}).

	For a genus $g$ surface $\Sigma_g$, the standard symplectic form $\w_\can=\dd\lambda_\can$ on $T^\ast\Sigma_g$ determines, for any metric, a Liouville structure on the associated cotangent disk bundle $D^\ast\Sigma_g$. The boundary of this domain is the cotangent sphere bundle $S^\ast\Sigma_g$ with its standard contact structure $(S^\ast\Sigma_g, \xi_\can)$. Whenever the surface $\Sigma_g$ is hyperbolic, McDuff's Anosov Liouville domain construction \cite{McD} gives a Liouville domain $I\times S^\ast\Sigma_g$ with boundary $(S^\ast \Sigma_g, \lambda_\can)\sqcup (-S^\ast\Sigma_g, \lambda_\pre)$, where $\lambda_\pre$ is the prequantization contact form on the circle bundle. We deduce the following fact from \Cref{thm:inf_spec_invts}.

\begin{theorem}\label{thm:cot_inf}
The ECH capacities of a cotangent disk bundle $(D^\ast \Sigma_g, \w_\can)$  over a surface of genus at least two are infinite.
\end{theorem}

\begin{proof}
	The ECH capacity $c_k^{\ECH}(D^\ast \Sigma_g, \w_\mathrm{can})$ is defined to agree with the $k$th ECH spectral invariant $c_k(S^\ast\Sigma_g, \lambda_\mathrm{can})$ of the boundary. McDuff's construction of an exact filling of $(S^\ast\Sigma_g,\lambda_\can)\sqcup (-S^\ast\Sigma_g, \lambda_{\pre})$ together with \Cref{thm:inf_spec_invts} implies $c_k(S^\ast\Sigma_g, \lambda_\mathrm{can})=\infty$. 
\end{proof}

\begin{remark}
	The statement that the ECH capacities of a cotangent disk bundle $D^\ast \Sigma_g$ are finite or infinite is independent of the choice of Riemannian or Finsler metric used to define the unit disk. Indeed, any disk bundle can be rescaled to fit in any other, so the capacities must be infinite for any tubular neighbourhood of the zero-section in $T^\ast\Sigma_g$ considered as a symplectic sub-domain. 
\end{remark}

\subsubsection{Anosov contact manifolds}
Even given an explicit Anosov flow, it may be difficult to tell which contact structures bound an Anosov Liouville domain. We study another family of contact structures coming from Anosov dynamics which always have infinite ECH spectrum.

\begin{definition}
	An \emph{Anosov contact 3-manifold} is a closed contact 3-manifold $(Y,\xi)$ admitting a contact form whose Reeb flow is oriented Anosov.
\end{definition}

The basic example of an Anosov contact manifold is the cotangent circle bundle $(S^\ast \Sigma_g, \xi_\can)$ of an oriented hyperbolic surface. More Anosov contact manifolds can be constructed explicitly thanks to a surgery technique of Foulon--Hasselblatt \cite{FH}. Broad existence results and construction methods also follow from the recent work of Marty \cite{Marty}. Several authors have used Floer theory and symplectic field theory to obstruct and describe these Reeb Anosov flows \cite{FHV, MP, CP}. Hozoori \cite{Hoz20} asked how the existence of a Reeb Anosov flow for a contact structure restricts its contact topology; the following result constitutes some progress in this direction. The statement involves a distinguished element $c(Y,\xi) \in ECH(Y,\xi)$ called the \emph{contact invariant}; its definition is given in \Cref{ssec:ECHprops}.

\begin{theorem}
\label{thm:AnoReeb}
	Suppose $(Y,\xi)$ is an Anosov contact 3-manifold. Then it has non-vanishing contact invariant, but the ECH spectrum of $Y$ is infinite for any contact form for $\xi$ and the ECH capacities of any weakly exact filling of $(Y,\xi)$ are infinite.
	\end{theorem}

	 ECH also possesses an endomorphism $U\!\!: ECH(Y,\xi)\to ECH(Y,\xi)$ called the \emph{$U$ map}, defined in \Cref{ssec:ECHprops}. Finiteness of the first element of the ECH spectrum corresponds exactly to the contact invariant lying in the image of the $U$ map. Thus, \Cref{thm:AnoReeb} can be reinterpreted as giving restrictions on when a contact structure $(Y,\xi)$ can admit a Reeb Anosov flow, in terms of the structure of ECH as an $\F_2[U]$-module.
	\begin{corollary}\label{cor:AnoReeb}
		A contact 3-manifold $(Y,\xi)$ can be Anosov only if $c(Y,\xi)$ is non-zero and not in the image of the $U$ map on $ECH(Y,\xi,0)$.
	\end{corollary}
As further motivation for \Cref{cor:AnoReeb}, a central problem in the modern theory of Anosov flows (which dates back at least to Smale \cite{Sma}) is their classification on a given 3-manifold up to \emph{orbit equivalence} (see \Cref{ssec:Ano} for a definition). One important class of Anosov flows are the skew $\R$-covered flows. Recent work of Barthelm\'{e}--Mann and Marty \cite{BarM,Marty} implies the classification of oriented skew $\R$-covered Anosov flows is equivalent to the classification of Anosov contact structures up to contactomorphism. By \Cref{cor:AnoReeb}, it suffices to restrict to contact structures with a very specific property of their contact invariant. 

\subsubsection{Hamiltonian Anosov flows}
Given a symplectic manifold $(X^{2n},\w)$ and a smooth closed hypersurface $Y^{2n-1} \subset X$, the one-dimensional distribution $\ker(\w|_{Y})\subset TY$ integrates to a \emph{characteristic foliation} of $Y$. This foliation has the following dynamical interpretation: given a smooth function $H\!\!: X\to \R$ with $Y$ as a regular level set, the flow of the associated Hamiltonian vector field $X_H$ is tangent to $Y$ and along $Y$ its flow lines are integral curves of the characteristic foliation. Thus, it makes sense to talk about the \emph{characteristic flow} of $Y$ (at least up to reparametrization). In his 1998 ICM address, Michael Herman posed the following question.
\begin{question}[Herman, 1998 ICM \cite{Her}]\label{q:Herm}
	Given a smooth compact Hamiltonian hypersurface $Y^{2n-1} \subset \R^{2n}$, can its characteristic Hamiltonian flow be Anosov?
\end{question} 

One motivation for studying Anosov Hamiltonian flows is a classical theorem of Newhouse \cite{New}, which implies that given a regular Hamiltonian hypersurface $H^{-1}(c)$ in a symplectic 4-manifold, in a $C^2$-open neighbourhood of smooth functions $\tilde{H}$ near $H$, a $C^2$-dense open set of hypersurfaces $\tilde{H}^{-1}(c)$ either has an elliptic orbit or its flow is Anosov. Hence Anosov Hamiltonian flows exactly obstruct the $C^2$-genericity of elliptic periodic orbits.

To the author's knowledge, Herman's question has not yet been addressed in the literature. Here we concern ourselves with the case $n=2$, since Anosov flows are best understood and known to exist plentifully in dimension three. First, we give an elementary construction of many such hypersurfaces.

\begin{proposition}\label{prop:AnoHamconstruction}
	Every symplectic 4-manifold admits infinitely many pairwise non-homeomorphic Hamiltonian hypersurfaces with an Anosov characteristic flow.
\end{proposition}
\begin{corollary}
	In any symplectic 4-manifold, the collection of Hamiltonian hypersurfaces with an elliptic orbit of their characteristic foliation is not $C^2$-Baire generic in the space of regular level sets of smooth Hamiltonians.
\end{corollary}

 All the hypersurfaces we will construct carry a non-orientable Anosov flow, and  so it is reasonable to modify Herman's question by demanding that the characteristic flow be oriented Anosov. In a spirit analogous to \Cref{thm:AnoReeb}, one can give a general obstruction to oriented Anosov Hamiltonian flows.

\begin{theorem}\label{thm:AnoHam}
Suppose that $(X,\w)$ is a four-dimensional weak Liouville domain with $c_1^{\ECH}(X,\w)<\infty$. Let $H:X\to \R$ be a smooth Hamiltonian and $M= H^{-1}(c)$ a compact regular level set. Then the characteristic Hamiltonian flow on $M$ cannot be oriented Anosov. 
\end{theorem}
Since compact domains in $\R^4$ have finite first ECH capacity, and every hypersurface in $\R^4$ is a level set of some function, this answers the oriented refinement of \Cref{q:Herm} in the negative.

\Cref{thm:AnoHam} can also be used to prove the generic existence of elliptic orbits on certain Hamiltonian hypersurfaces in Liouville domains with $c_1^{\ECH}<\infty$  (see \Cref{cor:elliptic}).

\subsubsection{Strong co-fillings and Anosov flows} We now observe a Floer-theoretic obstruction to symplectic co-fillings, and by proxy, Anosov flows.
\begin{definition}
	A connected contact manifold $(Y,\xi)$ is called \emph{strongly co-fillable} if there exists a contact manifold $(Y',\xi')$, non-empty but possibly disconnected, so that $(Y,\xi)\sqcup (Y',\xi')$ admits a connected strong filling.
\end{definition} 
\begin{remark}
Note that being strongly co-fillable is a more restrictive condition than being strongly fillable in dimension three, since any strong co-filling of $(Y,\xi)$ can have its other boundary component(s) capped off \cite{EH} to produce a strong filling of $(Y,\xi)$. The tight contact structure on $S^3$ gives an example of a strongly fillable but not strongly co-fillable	contact manifold, as was first shown by McDuff \cite{McD}.

To avoid confusion, we note there is also an older notion of \emph{strong semi-fillability} of $(Y,\xi)$, which is usually taken to mean $(Y,\xi)$ is one component of the boundary of a strong filling (that may or may not have disconnected boundary). The existence of symplectic caps \cite{EH} implies strong semi-fillability and strong fillability are equivalent in dimension three.
\end{remark}

Mitsumatsu \cite[p.\!\!\! 1420]{Mit} asks when a $3$-manifold can or cannot be realized as a boundary component of a symplectic manifold with disconnected convex boundary. Some previous obstructions to contact structures being co-fillable were found by Wendl \cite{Wen} using the notion of partially planar domains. We now give a Floer-theoretic obstruction to a 3-manifold admitting any strongly co-fillable contact structure. The Anosov Liouville construction then implies this is also an obstruction to Anosov flows. We phrase the result in terms of monopole Floer cohomology rather than ECH to emphasize its topological nature, independent of a choice of contact structure. Given an oriented  2-plane field $\xi$ on a closed 3-manifold $Y$ (for example a contact structure), one can associate to it a spin$^c$-structure $\mathfrak{s}_\xi$ (see \cite[\S 8.1]{Hut10}).

\begin{theorem}\label{cor:Umap}
	Suppose $Y$ is a closed oriented three-manifold so that the $U$ map on the `from' variant of monopole Floer cohomology for the spin$^c$-structure $\mathfrak{s}$
	\[U : \widehat{HM}^\ast(Y,\mathfrak{s})\to \widehat{HM}^\ast(Y,\mathfrak{s}) \quad \text{ satisfies $\ker(U) \subset \im(U)$. }\]
	Then any contact structure $(Y,\xi)$ on $Y$ with $\mathfrak{s}_\xi = \mathfrak{s}$ cannot be strongly co-fillable. In particular, if this holds for every spin$^c$-structure $\mathfrak{s}$ with $c_1(\mathfrak{s})=0$, then  $Y$ cannot carry an oriented Anosov flow.
\end{theorem}

\begin{example}Note the hypotheses of this theorem are satisfied for every spin$^c$-structure whenever the $U$ map is surjective, which is true for example on $S^3$,  $T^3$, and $S^2\times S^1$. More generally, it is true for any 3-manifold which admits a metric of positive scalar curvature \cite[Prop.\!\! 36.1.3]{KM}; in this case \Cref{cor:Umap} recovers a result of Lisca \cite{Lis}. 

The hypotheses also hold for any \emph{$L$-space}. Recall an $L$-space $Y$ is a rational homology sphere so that the reduced monopole Floer cohomology
\[HM^\ast(Y) := \im(j^\ast : \widehat{HM}^\ast(Y) \to \widecheck{HM}^\ast(Y)) \quad \text{vanishes.}\]
This condition implies  there is a surjective map of $\F_2[U]$-modules
\[p^\ast : \overline{HM}^\ast(Y, \mathfrak{s}) \to \widehat{HM}^\ast(Y,\mathfrak{s})\]
for each spin$^c$-structure $\mathfrak{s}$ on $Y$. Since $Y$ is a rational homology sphere, $c_1(\mathfrak{s})$ is torsion and \cite[Prop. 35.3.1]{KM} shows that 
\[\overline{HM}^\ast(Y,\mathfrak{s}) \cong \F_2[U^{-1}, U] \quad \text{as $\F_2[U]$-modules,}\] 
hence the $U$ map is surjective on $\widehat{HM}^\ast(Y,\mathfrak{s})$.
\end{example}

The obstruction to Anosov flows given in \Cref{cor:Umap} has been obtained independently by Hedden--Raoux--Van Horn-Morris \cite{HRVHM} in terms of Heegaard Floer homology. Using this obstruction, they find the first examples of hyperbolic 3-manifolds $Y$ with $b_1(Y)>0$ admitting no oriented Anosov flows (c.f. the discussion in \cite[p.\! 363]{Yaz}). Note, by Gabai's theorem, such examples have many taut foliations. 

Hedden--Raoux--Van Horn-Morris' work builds on a result of Lin \cite{Lin}, who makes use of the TQFT properties of monopole Floer theory to give an obstruction to co-orientable taut foliations on rational homology spheres. This is an enhancement of older work of Ozsv\'{a}th  and Szab\'{o} \cite{OS}, who prove that an $L$-space $Y$ admits no co-orientable contact structure with a weak co-filling of the form $I\times Y$ (and therefore, by Eliashberg--Thurston \cite{ET}, no co-orientable taut foliation). This result was part of the genesis of the \emph{$L$-space conjecture} (c.f. \cite{BGW,Juh}). 

Since Anosov flows give rise to taut foliations, Ozsv\'{a}th--Szab\'{o} and Lin's conclusions are stronger but require the extra assumption that $Y$ is a rational homology sphere. We note for example that their results do not apply to $T^3$; and indeed, $T^3$ admits a taut foliation and hence a weak co-filling, but it is well known (and our result gives another proof) that it cannot have a strong co-filling or Anosov flow.  One can recover Lin's result from \Cref{cor:Umap} by noting that when $Y$ is a rational homology sphere, every weak filling $I\times Y$ can be deformed to a strong filling \cite{Eli}.

\begin{remark}	The assumption that our Anosov flows are orientable is essential throughout. The geodesic flow on the cotangent circle bundle of a hyperbolic non-orientable surface is a non-orientable Reeb Anosov flow. By \Cref{prop:non-orient}, it has finite ECH spectrum, which contradicts \Cref{thm:AnoReeb} if one does not require orientability. 

  As noted by Baldwin--Sivek--Zung \cite[p.\!\!\!\!\! 2]{BSZ}, $56/3$-surgery on the pretzel knot $P(-2,3,7)$ gives rise to a hyperbolic $L$-space admitting an Anosov flow which is not transversely orientable. Because this is an $L$-space, it provides a counterexample to \Cref{cor:Umap} if one drops the oriented assumption. 
\end{remark}

\subsubsection{High-genus Lagrangians} Given a symplectic 4-manifold $(X,\w)$ and a closed embedded Lagrangian surface $L\subset X$, the Weinstein tubular neighbourhood theorem guarantees a neighbourhood of $L$ in $X$ is symplectomorphic to $(D^\ast L,\w_\can)$ for a small radius. Monotonicity of the ECH capacities under symplectic embeddings then implies 
\[c_k^{\ECH}(D^\ast L,\w_{\can})\leq c_k^{\ECH}(X,\w).\] 
Note that, unlike for say the case of Gutt--Hutchings capacities \cite[Rmk.\!\! 1.25]{GH}, we do not require that $D^\ast L \subset X$ is a Liouville embedding, or that $L$ is an exact Lagrangian for monotonicity to hold. We thus deduce the following corollary of \Cref{thm:cot_inf}.

\begin{corollary}
\label{cor:Lags}
	Suppose $(X,\w)$ is a symplectic 4-manifold with $c_1^\ECH(X,\w)<\infty$. Then $(X,\w)$ does not contain any closed embedded oriented Lagrangian submanifold of genus two or greater. 
\end{corollary}

While this provides an obstruction to the existence of high-genus Lagrangians in many cases, for existing examples where the ECH capacities are finite there is usually easier topological obstructions to these Lagrangians embeddings. However, in future work we will give many more examples of domains with finite first ECH capacity and complicated topology, so that the non-existence of high-genus Lagrangians cannot be deduced from elementary obstructions. 

 On the other hand, whenever high genus Lagrangians are already known to exist inside $(X,\w)$, \Cref{cor:Lags} implies its ECH capacities are infinite. This proves the ECH capacities are infinite in some new cases, including the standard symplectic form on the 4-torus or more generally any product symplectic structure on $\Sigma_g \times \Sigma_{g'}$ with $\min\{g,g'\}\geq 1$.

\subsubsection{Cotangent bundles of non-orientable surfaces}

We have seen that for an oriented surface $\Sigma_g$, the ECH capacities of its cotangent disk bundle $D^\ast \Sigma_g$ are finite if and only if $g=0$ or $1$. One might also be interested in the case of cotangent disk bundles of non-orientable surfaces. For $\R P^2$ and the Klein bottle, the capacities are known to be finite \cite{FR,MR}. We extend this to all non-orientable surfaces.
\begin{proposition}\label{prop:non-orient}
	For any closed non-orientable surface $N$, the ECH capacities of a cotangent disk bundle $(D^\ast N, \w_\can)$ are finite. 
\end{proposition}
We therefore have a complete classification of which closed two-manifolds have finite ECH capacities of their cotangent disk bundles. 

\subsubsection{Capacities from rational SFT and co-fillings}
Hutchings has sketched the construction of a sequence of \emph{RSFT capacities}: 
\[c_1^{\mathrm{RSFT}}(X,\w)\leq c_2^\mathrm{RSFT}(X,\w)\leq c_3^\mathrm{RSFT}(X,\w)\leq \cdots \leq \infty \]
associated to a symplectic manifold $(X,\w)$ of any dimension \cite{HutRSFT}. These are defined analogously to the ECH capacities in terms of a ``$q$-only" formulation of rational symplectic field theory, contingent on suitable rigorous foundations for the invariant. Because of their similar formal properties, a version of \Cref{thm:inf_spec_invts} can be proven for these capacities. 

\begin{proposition}\label{prop:rat}
Suppose $(X,\w)$ is a connected Liouville domain with disconnected boundary. Then its RSFT capacities are  infinite.\end{proposition}

This proposition of course applies to the Anosov Liouville domains. In particular, this implies that the RSFT capacities of the cotangent disk bundle over a surface of genus at least two are also infinite. Unlike the results for ECH, this can also be applied to certain domains in higher dimensions, including versions of the Liouville pair construction \cite{Gei6,MNW}.

\addtocontents{toc}{\protect\setcounter{tocdepth}{1}}
\subsection*{Outline}
In \Cref{sec:defs}, we recall necessary background on both ECH and Anosov flows, in hope that this paper will be accessible to a reader conversant in one subject but not the other. In \Cref{sec:proof}, we prove our main results, namely \Cref{thm:inf_spec_invts,thm:AnoReeb,thm:AnoHam,cor:Umap}. In \Cref{sec:constructions}, we generalize our results to the setting of hypertaut foliations, prove \Cref{prop:non-orient,prop:AnoHamconstruction}, discuss high-genus Lagrangians, and speculate on elementary capacities.  A construction of RSFT capacities and a proof of \Cref{prop:rat} are given in \Cref{app:rsft}.
	
\subsection*{Acknowledgements} I'd like to thank my advisor Michael Hutchings for his guidance and many enlightening discussions and thank Audrey Rosevear  and Rohil Prasad for very helpful conversations. The contents of this paper benefitted immensely from many lectures and discussions at the SLMath Spring 2026 program on Topological and Geometric Structures in Low Dimensions. I'd also like to thank Michael Hutchings, Josie Hlavinka, Siddhi Krishna, Rohil Prasad, and Audrey Rosevear for their comments on various earlier drafts of this work. Finally, I would like to acknowledge the support of the Natural Sciences and Engineering Research Council of Canada (NSERC) through PGS D-587425-2024. Finalement, je voudrais remercier le Conseil de recherches en sciences naturelles et en g\'{e}nie du Canada (CRSNG) de son soutien.  
\addtocontents{toc}{\protect\setcounter{tocdepth}{2}}

\section{Background on ECH capacities and Anosov flows}
\label{sec:defs}

\subsection{Embedded contact homology}
Embedded contact homology is a version of Floer homology and symplectic field theory for contact 3-manifolds. We give a brief recollection of its definition and properties; many more details can be found in Hutchings' compendium \cite{Hut14}. 

\begin{definition}
	Given a closed oriented 3-manifold $Y$, a 1-form $\lambda$ is called a \emph{contact form} provided $\lambda \wedge\dd\lambda>0$. The pair $(Y,\lambda)$ defines a \emph{(co-oriented) contact structure} given by the 2-plane distribution $\xi =\ker \lambda$, and a \emph{Reeb vector field} $R$ uniquely determined by the conditions $\iota_R \dd\lambda=0$ and $\iota_R \lambda=1$.
\end{definition}
 We will be interested in \emph{Reeb orbits}, which are periodic orbits $\gamma : \R/T\Z \to Y$ of the Reeb flow of $R$ modulo reparametrization of the domain. Associated to such an orbit $\gamma$ is a Poincar\'{e} return map $P_\gamma$, which is the symplectic linear map $\dd \psi^T_R: \xi_{\gamma(0)} \to \xi_{\gamma(0)}$ induced by the linearization of the time $T$ flow $\psi^T_R$ of $R$. We say $\gamma$ is \emph{non-degenerate} if $P_\gamma$ does not have one as an eigenvalue. We assume that $\lambda$ is \emph{non-degenerate}, meaning that all its Reeb orbits are non-degenerate; this is a generic condition achievable by a perturbation of $\lambda$. A non-degenerate orbit $\gamma$ may be either \emph{elliptic} if its eigenvalues $\lambda,\lambda^{-1}$ are unit complex numbers, or \emph{positively} (respectively \emph{negatively}) \emph{hyperbolic} if its eigenvalues are real and positive (resp. negative).

For $\lambda$ non-degenerate, $\Gamma \in H_1(Y;\Z)$, and a certain almost-complex structure $J$ on $\R\times Y$, we can then construct a chain complex $ECC_\ast(Y,\lambda, \Gamma, J)$. 
\begin{definition}
	An \emph{orbit set} for $(Y,\lambda)$ is a finite set of pairs $\{(\gamma_i, m_i)_{i=1}^n\}$, where $\gamma_i$ are distinct embedded Reeb orbits and $m_i$ are positive integer multiplicities. We say an orbit set is \emph{admissible} if $m_i=1$ whenever $\gamma_i$ is hyperbolic. 
\end{definition}
Let $ECC_\ast(Y,\lambda, \Gamma, J)$ be the free $\F_2$-module generated by all admissible orbit sets $\{(\gamma_i, m_i)_{i=1}^n\}$ which represent the homology class 
\[\Gamma = \displaystyle\sum_{i=1}^n m_i [\gamma_i] \in H_1(Y;\Z).\] 

To define a differential, we consider the \emph{symplectization} $\R_t \times Y$ equipped with the symplectic form $\w=\dd{(e^t \lambda)}$. An almost-complex structure $J$ on $\R\times Y$ is called \emph{admissible} if it satisfies $J(\pd_t)=R$, preserves the distribution $\xi$,  satisfies  $\dd\lambda(v, Jv)>0$ for non-zero $v \in \xi $, and is invariant under translation in the $\R$-coordinate.
\begin{definition}
	Given orbit sets $\alpha=\{(\alpha_i, m_i)\}$ and $\beta=\{(\beta_j,n_j)\}$, we define a \emph{$J$-holomorphic current} from $\alpha$ to $\beta$ as a finite set $\mathcal{C}=\{(C_k, d_k)\}$ of distinct irreducible somewhere-injective $J$-holomorphic curves $C_k$ in $\R \times Y$ (modulo biholomorphic reparametrization of the domain) and positive integer multiplicities $d_k$ so that as $t\to \pm \infty$, $\mathcal{C}$ \emph{converges as a current} to $\alpha$ and $\beta$ respectively. By this we mean that the positive asymptotics of the curves $C_k$ counted with multiplicities $d_k$ coincide with covers of the embedded orbits $\alpha_i$ and the number of times $\alpha_i$ appears counted with covering multiplicity agrees with $m_i$ (and identically at the negative ends for the $\beta_i$'s). To such a current, Hutchings associates an \emph{ECH index} $I(\alpha, \beta, [\mathcal{C}]) \in \Z$  which depends only on the relative homology class of $\mathcal{C}$ and its asymptotic orbit sets \cite{Hut14}. The ECH index  determines a relative grading; this can be reduced to a canonical $\Z/2\Z$ grading which measures the parity of the number of positive hyperbolic orbits in an orbit set. 
\end{definition}

Consider an admissible $J$ on $\R\times Y$. Let $\mathcal{M}_1(\alpha,\beta)$ denote the moduli space of all $J$-holomorphic currents from $\alpha$ to $\beta$ which have ECH index one and let $\widetilde{\mathcal{M}_1}(\alpha,\beta)$ denote this space modulo the identification of curves differing by translation in the $\R$-coordinate of $\R_t\times Y$. A combination of facts about 4-dimensional pseudoholomorphic curves guarantees every $\mathcal{C} \in \mathcal{M}_1(\alpha,\beta)$ consists of a (possibly empty) collection of trivial cylinders $\R\times \gamma$ and an irreducible embedded $J$-holomorphic curve of Fredholm index one. In particular, for a generic choice of admissible $J$, $\widetilde{\mathcal{M}_1}(\alpha,\beta)$ is a finite set. The ECH differential is defined to act on an admissible orbit set $\alpha$ by
\[\partial \alpha = \sum_{\beta} \# \widetilde{\mathcal{M}}_1(\alpha,\beta)\cdot \beta,\]
where the sum is taken over all admissible orbit sets $\beta$. Since we work with $\F_2$-coefficients, $\#\widetilde{\mathcal{M}_1}(\alpha,\beta)$ is simply the mod 2 count of points in the moduli space. 

Hutchings--Taubes \cite{HTOBG} show $\pd^2=0$ and so $ECC_\ast(Y,\lambda,\Gamma, J)$ is really a chain complex. Its homology, denoted $ECH_\ast(Y,\xi, \Gamma)$, is called \emph{embedded contact homology}. It turns out these homology groups depend only on the contact structure and not the contact form or almost-complex structure; this is due to Taubes \cite{Tau}. In fact, to the 2-plane field $\xi$ one associates a spin${}^c$-structure $\mathfrak{s}_\xi$ and Taubes shows, for $Y$ connected, that there is a canonical isomorphism 
\[ECH_\ast(Y,\xi, \Gamma) \cong \widehat{HM}^{-\ast}(Y, \mathfrak{s}_\xi + \mathrm{PD}(\Gamma))\]
identifying ECH with the ``from" flavour of monopole Floer cohomology, as defined by Kronheimer--Mrowka \cite{KM}. Here $\mathrm{PD}$ denotes Poincar\'{e} dualizing and we use the fact that the space of spin${}^c$-structures on $Y$ is an affine space modelled on $H^2(Y)$. There is a further  isomorphism
\[ECH_\ast(Y,\xi,\Gamma) \cong HF_\ast^+(-Y, \mathfrak{s}_\xi + \mathrm{PD}(\Gamma)) \]
 to the ``plus" version of Heegaard Floer homology of $-Y$ \cite{CoGH,KLT}.

Note that when $(Y,\xi) = (Y_1, \xi_1)\sqcup \cdots \sqcup (Y_n,\xi_n)$ is disconnected, the basic definitions along with a K\"{u}nneth formula guarantee
\begin{equation}\label{eq:ECHtensor}
	ECH_\ast(Y,\xi ,\Gamma) = ECH_\ast(Y_1,\xi_1,\Gamma_1)\otimes\cdots \otimes ECH_\ast(Y_n,\xi_n, \Gamma_n)
\end{equation}

for $\Gamma= (\Gamma_1,\ldots, \Gamma_n)\in H_1(Y)$.

\subsection{Structures and properties of ECH} \label{ssec:ECHprops}
\subsubsection{The contact invariant}
Given a contact 3-manifold $(Y,\xi)$, there is a distinguished element sensitive to the contact structure $c(Y, \xi) \in ECH(Y,\xi, 0)$ called the \emph{contact invariant}. This is the homology class represented by the empty orbit set $[\emptyset]$. Note that there are no holomorphic curves in $\R \times Y$ without positive ends by Stokes' theorem, so that $\emptyset$ is necessarily a cycle. It may be that $\emptyset$ is also a boundary, in which case $c(Y, \xi)$ vanishes (this holds for example if $\xi$ is overtwisted or contains Giroux torsion), but $c(Y, \xi)$ is always non-trivial when $(Y,\xi)$ is strongly fillable \cite{HutTQFT, Ech}.

There are also contact invariants defined by different means in monopole and Heegaard Floer theories and they all coincide under the  isomorphisms \cite{CoGH, KLT, Tau}.
\subsubsection{The $U$ map}
Assume $Y$ is connected. Given a generic point $y\in Y$, one can define a homomorphism
\[U_y: ECC_\ast(Y,\lambda, \Gamma, J) \to ECC_{\ast-2}(Y,\lambda, \Gamma, J)\]
which counts, in a manner analogous to the differential, ECH index two $J$-holomorphic currents in $\R\times Y$ constrained to pass through $(0,y)\in\R \times Y$. This is a chain map and given two generic points $y_1, y_2$, the induced maps $U_{y_1}, U_{y_2}$ are chain homotopic. Thus we obtain a canonical ``$U$ map"
\[U: ECH_\ast(Y,\xi, \Gamma) \to ECH_{\ast-2}(Y,\xi, \Gamma).\]
By Stokes' theorem, one always has $c(Y,\xi)\in\ker(U)$.

There are also $U$ maps on monopole and Heegaard Floer homologies  and the isomorphisms are maps of $\F_2[U]$-modules \cite{CoGH, KLT, Tau}.

If $Y$ is disconnected, there will be a $U$ map associated to each connected component. Under the isomorphism \eqref{eq:ECHtensor}, the $U$ map for the $i$th component of $Y_1\sqcup\cdots \sqcup Y_n$ is given by the $U$ map on $ECH(Y_i)$ tensor the identity in other components. 

\subsubsection{Cobordism maps}
ECH, like other Floer theories, has functorial properties similar to a topological quantum field theory. Suppose $(X,\w)$ is a compact connected exact symplectic cobordism from a contact 3-manifold $(Y_1,\lambda_1)$ to another $(Y_2,\lambda_2)$ with underlying contact structures $\xi_1$ and $\xi_2$ respectively. We can glue cylindrical ends to $X$ to form its \emph{symplectic completion}
\[ (\widehat{X}, \hat{\w}) = ((-\infty, 0]_t\times Y_2, \dd{(e^t \lambda_2)})\, \bigcup_{Y_2} \, (X,\w)\, \bigcup_{Y_1} \,([0,\infty)_t\times Y_1, \dd{(e^t\lambda_1)}).\]
 An almost-complex structure $J$ on $\widehat{X}$ is \emph{cobordism-compatible} if it restricts to an $\w$-compatible almost-complex structure on $X$ and on the cylindrical ends agrees with the restriction of admissible almost-complex structures $J_1$ and  $J_2$ on $\R\times Y_1$ and $\R\times Y_2$.

For any $A\in H_2(X,\pd X)$ and a generic cobordism-compatible almost-complex structure $J$, there is an associated $\Z/2$-module homomorphism
\[\phi(X,\w, A, J): ECC_\ast(Y_1,\lambda_1, \pd_1 A, J_1)\to ECC_\ast(Y_2,\lambda_2, \pd_2A, J_2)\]
where $\pd_{1} A$ and  $\pd_2 A$ denote the homology classes in $H_1(Y_1)$ and $H_1(Y_2)$ obtained from $A$ under the map $H_2(X,\pd X) \to H_1(\pd X)$ in the long exact sequence of the pair $(X,\pd X)$.  This cobordism map is a chain map inducing a canonical homomorphism on homology
\[\Phi(X,\w, A) : ECH_\ast(Y_1,\xi_1, \pd_1 A) \to ECH_\ast(Y_2, \xi_2, \pd_2 A).\] 
 This map satisfies the important properties
\[\Phi(X,\w, 0) [\emptyset] = [\emptyset] \qand \Phi(X,\w,A) \circ U_1 = U_2\circ \Phi(X,\w, A),\]
where $U_1$ and $U_2$ denote the $U$ maps associated to (any connected component of) $Y_1$ and $Y_2$ respectively.

Defining the map $\Phi$ directly by counting pseudoholomorphic curves runs into major difficulties because multiply-covered curves of negative ECH index sometimes appear in the cobordism. Instead, the maps $\Phi$ were constructed by Hutchings--Taubes \cite{HTCob} by first identifying ECH with monopole Floer homology and using the cobordism map defined for that theory. Strictly speaking, this only makes sense when $Y_1$ and $Y_2$ are connected because monopole Floer homology has only been defined for connected manifolds. But one can still define the cobordism map in general by directly perturbing the Seiberg--Witten equations (see \cite[Rmk.\!\ 1.12]{HTCob} and the proof of \cite[Thm.\!\ 2.3]{HutQ}). 

While it is not clear in general how holomorphic curves are counted by the cobordism map, Hutchings--Taubes at least prove the crucial property that the cobordism map is \emph{witnessed by holomorphic currents} \cite[Thm.\!\ 1.9]{HTCob}. More precisely, if 
\[\langle \phi(X,\w, A, J)\, \alpha, \beta\rangle \neq 0,\]
then there must be a broken $J$-holomorphic current (i.e. a finite sequence of $J_1$-holomorphic currents in $\R \times Y_1$, a $J$-holomorphic current in the completion of $X$, and a finite sequence of $J_2$-holomorphic currents in $\R \times Y_2$, so that consecutive currents have matching negative and positive asymptotic orbit sets) from $\alpha$ to $\beta$ in the homology class $A$ of total ECH index zero.

In unpublished work of Hutchings \cite{HutTQFT}, this story is extended to general strong symplectic cobordisms which need not be exact. That is, given a strong symplectic cobordism $(X,\w)$ between $(Y_1,\lambda_1)$ and $(Y_2,\lambda_2)$, one obtains an ECH cobordism map satisfying all the properties listed above.

\begin{remark}
	The cobordism maps in monopole Floer and ECH are compatible with the isomorphism between the two theories, essentially by definition. We warn the reader that it is not yet known whether the isomorphisms of monopole Floer and ECH with Heegaard Floer also respect cobordism maps. For example, it is only known that the Heegaard Floer cobordism map sends the contact invariant to the contact invariant in the case of Stein cobordisms.
\end{remark}

\subsubsection{Action filtration}
Associated to a Reeb orbit $\gamma$ of $(Y,\lambda)$ is its action $\mathcal{A}(\gamma) = \int_\gamma \lambda$, which measures the length/period of the orbit. We can then define the action of an orbit set as
\[\mathcal{A}\Big(\left\{\big(\gamma_i,m_i\big)_{i=1}^n\right\}\Big) =\sum_{i=1}^n m_i \int_{\gamma_i} \lambda.\]
This gives the chain complex $ECC(Y,\lambda, J, \Gamma)$ an $\R$-filtration; i.e., for $L\in \R$, we let $ECC^{L}(Y,\lambda, J, \Gamma)$ denote the sub-module generated by orbit sets with action strictly less than $L$.

The crucial observation is that Stokes' theorem implies if $\mathcal{C}$ is a $J$-holomorphic current in $\R \times Y$ from $\alpha$ to $\beta$ with admissible $J$, then 
\[\mathcal{A}(\alpha)-\mathcal{A}(\beta) = \int_{\mathcal{C}} \dd\lambda \geq 0.\]
In particular, the ECH differential is action decreasing and so $ECC^L(Y,\lambda,J, \Gamma)$ is a subcomplex whose homology we denote $ECH^L(Y,\lambda, \Gamma)$. While this depends heavily on $\lambda$, it turns out to be independent of $J$. This gives ECH the structure of a \emph{persistence module}, meaning for any $L<L'$, there is a homomorphism
\[\iota_{L,L'} : ECH^L(Y,\lambda, \Gamma) \to ECH^{L'}(Y,\lambda,  \Gamma)\]
induced by inclusion on the chain-level. These maps are functorial under composition and for all but a discrete set of $L$, the map $\iota_{L,L+\varepsilon}$ is an isomorphism when $\varepsilon$ is sufficiently small. We can recover the full ECH as a colimit 
\[ECH(Y,\xi,\Gamma) = \lim_{L\to\infty} ECH^L(Y,\lambda,\Gamma).\] 

Note because the $U$ map is also defined by counting $J$-holomorphic currents, it is compatible with the $\R$-filtration and induces maps
\[U^L: ECH^L(Y,\xi,\Gamma) \to ECH^L(Y,\xi ,\Gamma)\]
which satisfy $U^{L'}\circ \iota_{L,L'} = \iota_{L,L'} \circ U^L$.

Given an exact cobordism, because the ECH cobordism map is witnessed by holomorphic currents, Stokes' theorem again implies the ECH cobordism map is compatible with the $\R$-filtration. For strong but non-exact cobordisms this will not hold, since there may be an ``energy defect"  determined by how the symplectic form pairs with the homology class of the holomorphic current. However, we will only be interested in  the cobordism map $\Phi(X,\w,0 )$ associated to the zero homology class, in which case $\Phi$ is a map of $\R$-filtered modules. We state this as a proposition, which is a special case of what Hutchings shows.

\begin{proposition}[\cite{HutTQFT}]
Let $(X,\w)$ be a connected strong symplectic cobordism from $(Y_1,\lambda_1)$ to $(Y_2,\lambda_2)$, both closed non-degenerate contact 3-manifolds. For each $L> 0$, there is a canonical homomorphism
	\[\Phi^L(X,\w, 0): ECH^L(Y_1,\lambda_1,0)\to ECH^L(Y_2,\lambda_2,0)\]
	satisfying the following properties:
	\begin{enumerate}[label=\textbf{\arabic*.}, leftmargin = 0.8cm]
		\item  For each $L>0$,  \[\Phi^L(X,\w,0) [\emptyset] = [\emptyset].\]
		\item For $L'>L$, 
		 \[\Phi^{L'}(X,\w,0)\circ \iota_{L,L'} = \iota_{L,L'} \circ \Phi^L(X,\w,0).\]
		 \item For each $L>0$, 
		  \[\Phi^L(X,\w,0)\circ U_1^L = U_2^L \circ \Phi^L(X,\w,0),\]
		  where $U_1, U_2$ denote $U$ maps on any connected component of $Y_1$ and $Y_2$ respectively.
		 \item Given a generic cobordism-compatible almost-complex structure $J$, the cobordism map $\Phi^L(X,\w,0)$ is induced from a chain-level homomorphism 
		  \[\phi^L(X,\w,0, J): ECC^L(Y_1,\lambda_1,0, J_1)\to ECC^L(Y_2,\lambda_2,0, J_2)\] 
		  which is witnessed by holomorphic currents of ECH index zero representing the trivial homology class $0\in H_2(X,\pd X)$.
	\end{enumerate}
	\begin{proof}
		As with the exact case, the cobordism map is constructed by passing through Seiberg--Witten theory. The proof is given in \cite{HutTQFT}. We will just explain why this cobordism map associated to the zero homology class will preserve the $\R$-filtration. Suppose $\alpha$ is an orbit set of $Y_1$ and $\beta$ is an orbit set of $Y_2$ so that, for some generic cobordism-compatible $J$,
		\[\langle \phi(X,\w,0,J)\alpha,\beta\rangle \neq 0.\]
		We know $\phi$ is witnessed by a broken $J$-holomorphic current $\mathcal{C}$ whose total homology class $[\mathcal{C}]=0 \in H_2(X,\pd X)$. This means $\mathcal{C}$ (considered as a smooth chain in $X$) can be written as the sum of the boundary of a $3$-chain $\mathcal{B}$ contained in $X$ and $2$-chains $\mathcal{D}_1$ and $ \mathcal{D}_2$ contained in the cylindrical ends $\R\times Y_1$ and $\R \times Y_2$ respectively with positive (resp. negative) boundary at $\alpha$ (resp. $\beta$). Then we have
		\begin{align*}
			\int_\mathcal{C} \w &= \int_{\pd \mathcal{B}}\w + \int_{\mathcal{D}_1} \dd \lambda_1 + \int_{\mathcal{D}_2} \dd\lambda_2.
			\shortintertext{By Stokes' theorem,}
			&= \int_{\mathcal{B}} \dd\w + \int_{\alpha} \lambda_1 -\int_{\beta} \lambda_2\\
			&= \mathcal{A}(\alpha)-\mathcal{A}(\beta).
		\end{align*}  
		On the other hand, because $J$ is $\w$-compatible,
		\[\mathcal{A}(\alpha)-\mathcal{A}(\beta)=\int_{\mathcal{C}} \w \geq 0.\]
	\end{proof}
\end{proposition}

\subsubsection{ECH of an ellipsoid}
One explicit computation of ECH will be relevant to us. Consider the 4-dimensional ellipsoid for $a,b>0$ defined by
\[E(a,b) = \left \{(z_1, z_2)\in \C^2: \frac{\pi |z_1|^2}{a}+\frac{\pi |z_2|^2}{b} \leq 1\right\}. \]
This inherits a natural symplectic form $\w_\std = \dd x_1\wedge \dd y_1 + \dd x_2\wedge \dd y_2$ restricted from $\C^2$. The Liouville primitive
\[\lambda_\std = \frac{1}{2}\left(x_1 \dd{y_1} -y_1\dd{x_1} + x_2\dd{y_2} - y_2\dd{x_2}\right)\] 
restricts to a contact form on the boundary of $E(a,b)$ inducing the standard tight contact structure $\xi$ on $S^3$.

The Reeb vector field on $\pd E(a,b)$ can be written explicitly in polar coordinates on $\C^2$ as
\[R = \frac{2\pi}{a}\pdv{\theta_1} +\frac{2\pi}{b}\pdv{\theta_2}.\]
When $a/b\notin \Q$, $R$ has exactly two embedded orbits which we denote $\gamma_1$ and $\gamma_2$. The orbits are confined to the planes $\C\times \{0\}$ and $\{0\}\times \C$  and have periods $a$ and $b$ respectively; they are both non-degenerate and elliptic.

We can thus label the generators of $ECC(\pd E(a,b), \lambda_\std, 0)$ as $\gamma_1^{m_1}\gamma_2^{m_2}$ for any pair of non-negative integers $(m_1,m_2)$ corresponding to the orbit set $\{(\gamma_1,m_1), (\gamma_2,m_2)\}$. The ECH differential vanishes because all the generators live in even gradings, and so $ECH(\pd E(a,b),\lambda_\std ,0)$ is the free $\Z/2$-module generated by orbit sets $\gamma_1^{m_1}\gamma_2^{m_2}$. The action of an orbit set will be
\[\mathcal{A}(\gamma_1^{m_1}\gamma_2^{m_2}) = m_1a + m_2b.\]
These numbers are distinct for distinct pairs $(m_1,m_2)$ because $a$ and $b$ are rationally linearly independent. One finds, via Seiberg--Witten theory or a holomorphic curve computation \cite{Hut14}, that the $U$ map acts on a generator $\gamma_1^{m_1}\gamma_2^{m_2}$ by sending it to the generator of next smallest action. It sends the empty orbit set $[\emptyset]= c(S^3,\xi)$ to zero.

\subsection{ECH spectrum and capacities}
To a non-degenerate contact 3-manifold $(Y,\lambda)$, Hutchings \cite{HutQ} associates an increasing sequence of numerical invariants
\[0= c_0(Y,\lambda) < c_1(Y,\lambda) \leq c_2(Y,\lambda) \leq \cdots \leq +\infty.\]
\begin{definition}
	The $k$th element of the \emph{ECH spectrum} is defined by
\[c_k(Y,\lambda) := \inf \left\{L >0\,  \Big |\,[\emptyset] \in \im\Big(U^k: ECH^L(Y,\lambda, 0)\to ECH^L(Y,\lambda, 0)\Big) \right\}.\]
When $Y$ is disconnected, this means that $[\emptyset]$ should be in the image of any $k$-fold composition of $U$ maps associated to the different components of $Y$. One observes
\begin{equation}\label{eq:specdisj}
c_k((Y_1,\lambda_1)\sqcup (Y_2,\lambda_2)) = \max_{i+j=k} \left(c_i(Y_1,\lambda_1)+c_j(Y_2,\lambda_2)\right).	
\end{equation}

We may extend the definition to degenerate contact 3-manifolds $(Y, \lambda)$. Let $\{\lambda_n\}_{n=1}^\infty$ be a sequence of non-degenerate contact forms on $Y$ which $C^\infty$-converge to $\lambda$. Then we set $c_k(Y,\lambda)$ to be the limit of the sequence $c_k(Y,\lambda_n)$ as $n\to \infty$. This is independent of the choice of approximating non-degenerate sequence.
\end{definition}

The ECH spectrum are spectral invariants in the sense that each number $c_k(Y,\lambda)$ is equal to the action of some (null-homologous) orbit set of $(Y,\lambda)$. Note that the $k$th element of the ECH spectrum is finite if and only if the contact invariant $c(Y,\xi)$ is in the image of $U^k$. In particular, finiteness of the ECH spectrum depends only on the contact structure while the precise values depend heavily on the choice of contact form. 

We can then associate to a symplectic 4-manifold $(X,\w)$ a sequence of capacities:
\[0= c_0^{\mathrm{ECH}}(X,\w) < c_1^{\mathrm{ECH}}(X,\w) \leq c_2^{\mathrm{ECH}}(X,\w) \leq \cdots \leq +\infty.\]
\begin{definition} 
Suppose $(X,\w)$ is a weak Liouville domain with contact-type boundary $(Y,\lambda)$. We  define its \emph{$k$th ECH capacity} to be
\[c_k^{\mathrm{ECH}}(X, \w) := c_k(Y,\lambda).\]
If $(X,\w)$ is  an arbitrary symplectic 4-manifold, we declare its $k$th ECH capacity to be the supremum of the $k$th ECH capacity of any weak Liouville domain $(X',\w')$ which symplectically embeds into $(X,\w)$.
\end{definition} 

We list some important properties of the ECH capacities.

\begin{proposition}[\cite{HutQ, Weyl}]
	The ECH capacities have the following properties.
	\begin{enumerate}[label=\textbf{\arabic*.}, leftmargin = 0.8cm]
		\item Given a symplectic embedding $(X,\w) \hookrightarrow (X',\w')$, the ECH capacities are monotone: 
		\[ c_k^{\mathrm{ECH}}(X,\w) \leq c_k^{\mathrm{ECH}}(X',\w')\quad \text{ for all $k$.} \]
		\item The capacities scale linearly with symplectic area: 
		\[\text{for any $s>0$,} \quad c_k^{\mathrm{ECH}}(X,s\w) = s c_k^{\mathrm{ECH}}(X,\w) \quad \text{for all $k$.}\]
		\item The ECH capacities satisfy a disjoint union property:
		\[c_k^{\ECH}\left(\coprod_{i=1}^n (X_i,\w_i)\right) = \max_{k_1+\cdots +k_n=k} \sum_{i=1}^n c_{k_i}^\ECH(X_i,\w_i).\]
		\item The ECH capacities satisfy a ``Weyl Law": if $(X,\w)$ is a weak Liouville domain whose ECH capacities are finite, then 
		\[\lim_{k\to \infty} \frac{c_k^{\mathrm{ECH}}(X,\w)^2}{4k} =  \mathrm{vol}(X,\w).\]
	\end{enumerate}
\end{proposition}

\subsection{Anosov and Reeb Anosov flows} \label{ssec:Ano}
We will recall the essential (for us) aspects of Anosov flows and their interaction with contact geometry. Many more details of this theory are explained by Fisher--Hasselblatt \cite{FHbook} and in recent surveys of Massoni and Potrie \cite{Mas2,Pot}. 

Consider a positive hyperbolic matrix $A\in \SL(2,\R)$ with real eigenvalues $\lambda > 1>\lambda^{-1}$. Under iteration of $A$, vectors in $\R^2$ will exponentially expand along one eigendirection and exponentially contract along the other. Given a hyperbolic surface $\Sigma$ and the corresponding geodesic flow on its cotangent circle bundle $S^\ast \Sigma$, a similar phenomenon occurs. As one moves along a geodesic, in a plane transverse to the flow, one sees exponential expansion and contraction in complementary directions; Anosov flows generalize this behaviour. 

\begin{definition}
	An \emph{Anosov flow} on a 3-manifold $Y$ is a flow $\psi^t$ generated by a smooth non-vanishing vector field $X$ so that there is a  $\psi^t$-invariant splitting
\begin{equation}\label{eq:Ano}
	TY = \langle X \rangle\oplus  E^{ss}\oplus  E^{uu}
\end{equation}
for which there exists $C, \mu > 0$ so that for every $v \in E^{ss}$ and $w\in E^{uu}$,
\begin{equation}\label{eq:Ano2}\| \psi^t_\ast v\| \leq  C e^{- \mu t} \|v\| \qand \|\psi^t_\ast w\|\geq C^{-1}e^{\mu t}\|w\| \quad \text{for all $t\geq 0$.}
\end{equation}	

The distributions $E^s = \langle X \rangle\oplus E^{ss}$ and $E^u= \langle X \rangle\oplus E^{uu}$ are uniquely integrable to 2-dimensional foliations $\mathcal{F}^s$ and $\mathcal{F}^u$ called the (weak) \emph{stable and unstable foliations}. These foliations are leafwise $C^\infty$-smooth, but the transverse regularity as an assemblage of leaves is usually only $C^1$-differentiable. The weak foliations are also always \emph{taut}, meaning for every $p\in Y$ there is a closed loop $\gamma:S^1\to Y$ based at $p$ and transverse to the foliation.

The Anosov flow $X$ is \emph{oriented Anosov} if $E^{ss}$ and $E^{uu}$ are orientable, in which case we always pick orientations for the invariant splitting \eqref{eq:Ano} to be compatible with the orientation of $Y$. We will assume all our Anosov flows are oriented unless otherwise stated. Given a non-orientable Anosov flow, one can always pass to a double cover of $X$ to obtain an oriented flow.
\end{definition}

The constant $C$ in equation \eqref{eq:Ano2} will depend on a choice of metric $g$ used to define the norm $\|\cdot\|$ on tangent vectors. Given any metric $g$, one can average it over the Anosov flow $\psi^t$ for a sufficiently long timescale to produce a new metric for which we may take $C=1$ \cite{Mit}. Such a metric is called \emph{adapted}. 

An important motivation for defining Anosov flows is that they are \emph{structurally stable}: given a vector field $X$ generating an Anosov flow, $C^1$-small perturbations of $X$ also generate an Anosov flow and one which will be orbit-equivalent to $X$ \cite{FHbook}. Here, we say that flows on $Y$  are \emph{orbit equivalent} if there is a homeomorphism of $Y$ mapping (oriented but unparametrized) orbits to orbits. A central problem in the study of Anosov flows is the \emph{finiteness conjecture}, which asserts every closed 3-manifold can only admit finitely many Anosov flows up to orbit equivalence.

If $Y$ has an Anosov flow, then its universal cover $\tilde{Y}$ is homeomorphic to $\R^3$ and the weak stable and unstable foliations $\mathcal{F}^u$ and $\mathcal{F}^s$ lift in $\tilde{Y}$ to foliations $\tilde{\mathcal{F}}^u$ and $\tilde{\mathcal{F}}^s$ with leaves homeomorphic to $\R^2$ \cite{FHbook}. 
\begin{definition}
	The \emph{orbit space} $\mathcal{O}_\psi$ of an Anosov flow $\psi$ on $Y$ is the quotient of the universal cover $\tilde{Y}\cong \R^3$  by the lifted flow. Barbot \cite{Bar} and Fenley \cite{Fen} show that $\mathcal{O}_\psi$ is homeomorphic to a plane, bifoliated\footnote{This means  $\mathcal{O}_\psi$ possesses two foliations which meet transversely everywhere.} by the lifts $\tilde{\mathcal{F}}^u$ and $\tilde{\mathcal{F}}^s$ modulo $\psi$. 
	
	The Anosov flow $\psi$ is \emph{$\R$-covered} if the space of leaves of one (equivalently both) of the foliations $\tilde{\mathcal{F}}^u$ or $\tilde{\mathcal{F}}^s$ is Hausdorff (equiv.\!\ homeomorphic to $\R$). Such flows are always \emph{transitive}, meaning they have a dense orbit. By Barbot and Fenley, the orbit space $\mathcal{O}_\psi$ of an $\R$-covered Anosov flow $\psi$ considered as a bi-foliated plane can either be homeomorphic to $\R^2$ with foliations by vertical and horizontal lines, or the diagonal strip $\{(x,y)\in \R^2: |x-y|<1\}$ with foliations by vertical and horizontal line segments. In the former case, we call $\psi$ a \emph{product Anosov flow} and it must be orbit equivalent to a suspension flow (see \Cref{ex:Anos}). In the latter case, we call $\psi$ a \emph{skew Anosov flow}.
	 
\end{definition}

\begin{definition}
	A \emph{Reeb Anosov flow} on $Y$ is an Anosov flow which is generated by the Reeb vector field of a contact form on $Y$. We say a contact structure $(Y,\xi)$ is an \emph{Anosov contact structure} if $\xi$ has a contact form whose Reeb flow is Reeb Anosov for an oriented Anosov flow. All Reeb Anosov flows are $\R$-covered and skew \cite{FHbook}.
\end{definition}
Recent work of Cristofaro-Gardiner--Prasad \cite{CGP} implies that, provided $c_1(\xi)$ is torsion, a Reeb flow for $(Y,\xi)$ is Anosov if and only if the closure in $Y$  of its set of periodic orbits is uniformly hyperbolic.

\begin{example}\label{ex:Anos}
The following are the two most important examples of Anosov flows.
\begin{enumerate}[label=\textbf{\arabic*.}, leftmargin = 0.8cm]

\item Geodesic flow on the cotangent circle bundle $S^\ast \Sigma$ of any surface of negative curvature defines an Anosov flow. The flow will be oriented when $\Sigma$ is oriented. This is a Reeb Anosov flow for the standard contact form on the cotangent circle bundle, hence it is $\R$-covered and skew.
\item Any hyperbolic matrix $A\in \mathrm{SL}(2,\Z)$ defines an area-preserving diffeomorphism of the 2-torus $T^2\cong\R^2/\Z^2$. Let 
\[M_A= [0,1]_t \times T^2\Big/(0,Ax)\sim (1, x)\] 
denote the mapping torus of $A$. The suspension flow induced by the vector field $\pd_t$ on $M_A$ is Anosov. The flow will be oriented exactly when the eigenvalues of $A$ are positive, or equivalently $\tr(A)>2$. The same holds for the suspension flow of non-linear Anosov diffeomorphisms of $T^2$. These flows are $\R$-covered product flows. That means they are not skew and hence never Reeb Anosov.
\end{enumerate}
\end{example}
\begin{definition}
An Anosov flow is \emph{algebraic} if a finite cover of the flow is either hyperbolic geodesic flow or the suspension flow of an Anosov diffeomorphism of $T^2$. 
\end{definition}

\begin{remark}\label{rmk:taut}
From the structure of the orbit space, we see that for any Anosov flow on $Y$, the leaves of the weak foliations $\mathcal{F}^u$ and $\mathcal{F}^s$ must $\pi_1$-inject into $Y$. The Anosovity condition implies that each leaf contains at most one closed simple orbit of the Anosov flow and moreover that such an orbit (and its covers) can never be contractible in the leaf and hence in $Y$. Thus (Reeb) Anosov flows have no contractible closed orbits. Three-dimensional contact structures with such a Reeb flow are called \emph{hypertight} and are in particular tight (in fact universally tight).

	Since $\mathcal{F}^u$ and $\mathcal{F}^s$ are taut foliations, Novikov's theorem implies they can never be transverse to a contractible closed curve. Hence if a contact structure admits a Reeb flow transverse to these foliations it must also be hypertight. 
\end{remark}

Reeb Anosov flows have been studied in many cases, including by Foulon--Hasselblatt--Vaugon \cite{FHV}, Macarini--Paternain \cite{MP}, and Barthelm\'{e}--Mann \cite{BarM}. Recently, Chaidez--Pan \cite{CP} introduced the notion of pseudo-Anosov Reeb flows and pseudo-Anosov contact structures.

In light of \Cref{thm:AnoReeb}, we elaborate a bit on the abundance of Reeb Anosov flows. Many Reeb Anosov flows can be built using a surgery construction described by Foulon--Hasselblatt \cite{FH} and studied further by Foulon--Hasselblatt--Vaugon \cite{FHV}. In this procedure, one begins with a Reeb Anosov flow on $(Y,\xi)$ and a Legendrian knot $L\subset Y$ transverse to $E^{ss}$ and $E^{uu}$. One can then perform Dehn surgery on $L$ for certain slopes so that the resulting three manifold also has a Reeb Anosov flow. Such knots $L$ can be found in $(S^\ast\Sigma_g, \xi_{\can})$ as a section of the circle bundle over a closed geodesic in $\Sigma_g$ (this subsumes the Handel--Thurston surgery construction).  These surgeries can produce Reeb Anosov flows not orbit equivalent to an algebraic flow, including Reeb Anosov flows on hyperbolic 3-manifolds. Salmoiraghi \cite{Sal} has generalized this operation to show that every Anosov contact manifold admits infinitely many Legendrian knots along which $1/q$-Dehn surgery for $q\in \N$ always has a Reeb Anosov flow. 

Marty \cite{Marty} shows that an Anosov flow is $\R$-covered and skew if and only if it is orbit equivalent to a Reeb Anosov flow, so the zoology of Reeb Anosov flows is at least as complex as that of skew $\R$-covered Anosov flows. Together with work of Barthelm\'{e}--Mann \cite{BarM}, one has that two skew $\R$-covered Anosov flows are orbit equivalent if and only if their associated Reeb Anosov flows are associated to contactomorphic contact structures.  In turn, skew Anosov flows admit a group-theoretic classification due to Thurston \cite{Thu-fol}. 

We also remark that structural stability implies that if $\lambda$ is a contact form on $Y$ whose Reeb flow is Anosov, then for $C^2$-small functions $f: Y\to \R$, the contact form $e^f\lambda$ also has an Anosov Reeb flow. These two Reeb flows are orbit equivalent, but their quantitative dynamics may look quite different; for example, the ECH spectral invariants will certainly vary with $f$ (consider \cite{Weyl, Iri}).

Because there are only finitely many tight contact structures on a 3-manifold which are strongly fillable, there can only be finitely many Anosov contact structures on a given manifold (this implies the finiteness conjecture for $\R$-covered Anosov flows). However, Barthelm\'{e}--Mann--Bowden \cite[Thm.\! A.7]{BarM} show that Foulon--Hasselblatt surgery can produce hyperbolic 3-manifolds with at least any arbitrarily large finite number of non-contactomorphic Anosov contact structures. 

\subsection{Anosov Liouville domains} \label{ssec:AL}
To apply \Cref{thm:inf_spec_invts}, we are interested in symplectic manifolds that have disconnected convex boundary. So far, the richest source of examples known in all dimensions comes from the following structure.
\begin{definition}
An \emph{(exponential) Liouville pair} $(\alpha_1, \alpha_2)$ on $Y$ is a pair of oppositely oriented contact forms so that   
\[\lambda = e^{t} \alpha_1 + e^{-t} \alpha_2\]	
defines a Liouville 1-form on $\R_t \times Y$. 
\end{definition}

Work of Eliashberg--Thurston \cite{ET} and Mitsumatsu \cite{Mit} uncovered a close connection between Liouville pairs and Anosov flows. This relationship has been fleshed out in recent years by work of Hozoori and Massoni \cite{Hoz, Mas, Mas2}. 

For our purposes, the main upshot of this theory is that given an oriented Anosov flow on a 3-manifold, we can always associate to it a Liouville pair. We briefly review the construction as detailed in \cite{Hoz, Mas}. 

Choose a metric $g$ on $Y$ so that the splitting \eqref{eq:Ano} is orthogonal, and by time-averaging we may assume $g$ is adapted. Let $e_s, e_u$ be the oriented unit vector fields tangent to $E^{ss}$ and $E^{uu}$. Because $g$ is adapted, and the splitting \eqref{eq:Ano} is invariant, there are functions $r_s<0< r_u$ so that
\[\mathcal{L}_X e_s = -r_s e_s \qand \mathcal{L}_X e_u = -r_u e_u.\]
Take $\alpha_s$ and $\alpha_u$ to be the $1$-forms $g$-dual to $e_s$ and $e_u$ respectively. Then
\begin{equation}\label{eq:pair}
\mathcal{L}_X\alpha_s = r_s\alpha_s \qand \mathcal{L}_X \alpha_u = r_u\alpha_u.	
\end{equation}
Now set
\[\alpha_+ = \alpha_u-\alpha_s \qand \alpha_-=\alpha_u + \alpha_s.\]
From \eqref{eq:pair}, we see that these are both contact forms, $\alpha_+$ positively oriented and $\alpha_-$ negatively oriented. These 1-forms are only $C^1$, but can be perturbed to be smooth without altering their essential properties \cite{Hoz, Mas}; we use the same notation for these smooth approximations. See \Cref{img:Anosov} for a depiction of the ``bi-contact structure" induced by $\alpha_\pm$. One checks that 
\[\lambda = e^t\alpha_+ + e^{-t}\alpha_-\]
defines a Liouville 1-form on $\R\times Y$, so that $(\alpha_+,\alpha_-)$ is a Liouville pair. Note that the induced contact structures $\xi_\pm = \ker \alpha_\pm$ have transverse intersection along the direction of the Anosov flow. Moreover, the Reeb vector fields of $\alpha_\pm$ are transverse to the foliations $\mathcal{F}^{s}$ and $\mathcal{F}^u$. In particular by \Cref{rmk:taut}, the Reeb flows of $\alpha_\pm$ have no contractible orbits and the contact structures $\xi_\pm$ must be hypertight. Additionally, both contact structures are tangent to the Anosov flow, hence as plane fields they have non-vanishing sections and their Chern classes $c_1(\xi_\pm)$ are zero.

\begin{figure}[h]
\centering
\begin{minipage}{.5\textwidth}
\centering
	\includegraphics[scale=0.78]{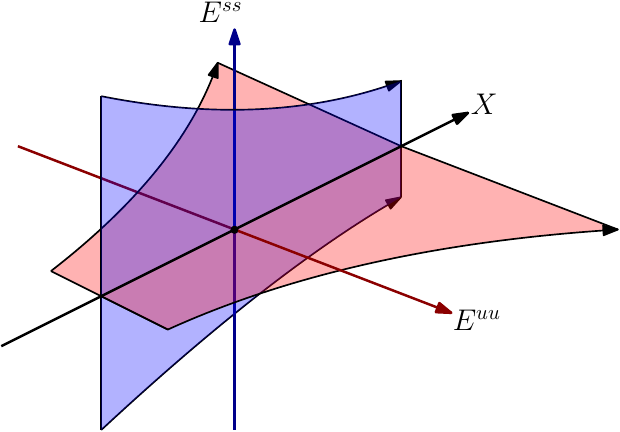}
\end{minipage}%
\begin{minipage}{.5\textwidth}
\centering
		\includegraphics[scale=0.78]{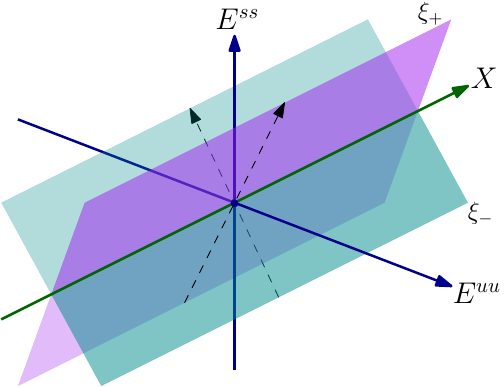}
\end{minipage}
	\caption{On the left, an illustration of the stable and unstable directions of the Anosov splitting and how vectors evolve under the Ansov flow $X$. On the right, an illustration of the bi-contact structure $\xi_\pm=\ker\alpha_\pm$ determined by the Anosov flow within a tangent space. Everything is oriented with respect to the right hand rule.}
	\label{img:Anosov}
\end{figure}

\begin{remark}
	The Anosov Liouville domains are of interest partly because they were the first examples produced of Liouville domains which are not Stein or Weinstein.\footnote{A Weinstein $2n$-manifold always has the homotopy type of an (at most) $n$-dimensional CW complex and so Weinstein domains for $n\geq 2$ have connected boundary.} The Floer homology of these domains has recently been investigated by Cieliebak--Lazarev--Massoni--Moreno \cite{CLMM} and it is shown to be an invariant of the Anosov flow up to orbit equivalence by Bowden--Massoni \cite{BM}. In particular, one can view these domains as a means to associate algebraic and Floer-type invariants to Anosov flows. The Fukaya category of the Anosov Liouville domains has some properties unlike Weinstein manifolds \cite{CLMM} and it remains an interesting question what extra symplectic flexibility properties these domains carry relative to the Weinstein setting.
\end{remark}

For our favourite Anosov flows from \Cref{ex:Anos}, we give an explicit description of their associated Anosov Liouville domains. These example are due to McDuff \cite{McD} and Geiges \cite{Gei} and predate the general construction of Mitsumatsu.

In McDuff's case \cite{McD}, one begins with the cotangent bundle $T^\ast \Sigma_g$ over a surface of genus $g\geq 2$. Consider a hyperbolic metric on the surface and let $\sigma$ be the associated volume form on $\Sigma_g$. One can define a ``magnetic deformation" of the symplectic form on the cotangent bundle to produce a new symplectic manifold
\[(T^\ast \Sigma_g, \w = \dd \lambda_{\can} + 2\pi p^\ast \sigma).\]
Here $\lambda_{\can}$ is the tautological Liouville $1$-form on $T^\ast \Sigma_g$ and $p: T^\ast\Sigma_g \to \Sigma_g$ is the projection map. Recall $\lambda_\can$ is a contact 1-form for the standard contact structure $\xi_{\can}$ on $S^\ast\Sigma_g$. By Chern--Weil theory and Gauss--Bonnet, $2\pi p^\ast \sigma$ is the curvature $\dd\lambda_\pre$ of a connection 1-form $\lambda_{\pre}$ on the principal $S^1$-bundle $-S^\ast\Sigma_g$. This is a contact 1-form for the prequantization contact structure $\xi_{\pre}$ on $-S^\ast\Sigma_g$ with Reeb vector field an infinitesimal generator of the $S^1$-action.

The tautological form $\lambda_{\can}$ scales proportionally to the radius while $\lambda_{\pre}$ does not, so that the contribution of $\dd\lambda_\can$ dominates at a  large radius while the magnetic field $\dd\lambda_\pre$ dominates at a small radius. One concludes the annulus bundle 
\[I\times S^\ast \Sigma_g \cong A^\ast \Sigma_g :=\{(x,v)\in T^\ast\Sigma_g : \varepsilon \leq \|v\| \leq \varepsilon^{-1}\} \]
is a Liouville domain for $\varepsilon>0$ sufficiently small with boundary 
\[(S^\ast\Sigma_g,\xi_\can)\sqcup (- S^\ast\Sigma_g, \xi_\pre).\]

For Geiges' example \cite{Gei}, we begin with the mapping torus $M_A$ of a positive hyperbolic matrix  $A\in \mathrm{SL}(2,\Z)$, considered as an Anosov diffeomorphism of $T^2$. We set $v_{\pm}$ to be  the eigenvectors of $A$ with eigenvalues $e^{\pm \mu}$ for $\mu>0$. On $\R^3_{x,y,z}$ there are a pair of contact forms
\[\alpha_{\pm} = e^{\mu z}\dd{v_+} \mp e^{-\mu z} \dd v_-.\]
There is then a Liouville 1-form on $\R_t\times \R^3$ given as  $\lambda= e^{t} \alpha_++ e^{-t}\alpha_-.$
We can quotient $\R^3$ by the relations
\[(x+n,y+m,z)\sim (x,y,z) \qq{for} n,m\in\Z \qand (A(x,y),z) \sim (x,y,z+1). \]
 Note that the quotient is the mapping torus $M_A$ and the $1$-forms $\alpha_{\pm}$ descend to the quotient, defining contact structures 
 \begin{equation}\label{eq:maptor}
\xi_\pm=\ker\alpha_\pm = \langle \pd_z, \pm e^{-\mu z}v_++e^{\mu z}v_-\rangle \quad \text{on $\pm M_A$.}
\end{equation}
 Hence $\lambda$ also descends to a Liouville form on $\R\times M_A$ and gives an exact filling of $(M_A,\xi_+)\sqcup (-M_A,\xi_-)$.

A unifying framework for these two and all other algebraic Anosov Liouville domains was observed by Geiges \cite{Gei} and Mitsumatsu \cite{Mit}. Let $G$ be either of the Lie groups $\widetilde{\mathrm{SL}}(2,\R)$ or $\mathrm{Sol}^3$. Here $\widetilde{\mathrm{SL}}(2,\R)$ is the universal cover of $\mathrm{SL}(2,\R)$ and $\mathrm{Sol}^3$ is an extension of $\R^2_{x,y}$ by $\R_z$ with $z$ acting as $(x,y) \mapsto (e^z x, e^{-z} y)$. These correspond to two of the  Thurston geometries. In either case, the structure of its Lie algebra implies $G$ has left-invariant vector fields $X_1,X_2,X_3$ satisfying
\[[X_1, X_3]=X_2\qand [X_2,X_3]=X_1.\]
Let $\alpha_i$ be left-invariant 1-forms dual to $X_i$ for $i=1,2,3$. Then $\alpha_1$ and  $\alpha_2$ are respectively positive and negative contact forms on $G$. They descend to the quotient $M=\Gamma \backslash G$ under any co-compact discrete subgroup $\Gamma \subset G$ acting by left-multiplication. One checks
\[\lambda = e^t \alpha_1 + e^{-t} \alpha_2\]
defines a Liouville form on $\R_t \times M$. 
For $G= \mathrm{Sol}^3$, the compact quotients $\Gamma \backslash G$ recover the mapping tori $M_A$ and mapping tori which they finitely cover. For $G= \widetilde{\mathrm{SL}}(2,\R)$, its compact quotients are the cotangent circle bundles $S^\ast\Sigma_g$ and more generally certain Seifert-fibred 3-manifolds which are circle bundles over a hyperbolic orbifold base. One family of examples are the Brieskorn spheres $\Sigma(2,3,6n+5)$ for $n\geq 1$; these were studied in our current context by Bowden \cite{Bow}. 

Note the following corollary of \Cref{thm:inf_spec_invts} applied to  these explicit examples of algebraic Anosov Liouville domains.
\begin{corollary}\label{cor:inf_spec}The following hold.
\begin{enumerate}[label=\textbf{\arabic*.}, leftmargin = 0.8cm, itemsep=0.2cm]
	\item For the unit cotangent circle bundle $S^\ast \Sigma_g$ over a closed surface of genus $g\geq 2$ with the canonical contact structure $\xi_{\can}$, the ECH spectrum is infinite. 
	\item For the circle bundle $-S^\ast \Sigma_g$ of Euler number $2-2g$ over a closed surface of genus $g\geq 2$ with the prequantization contact structure $\xi_{\pre}$, the ECH spectrum is infinite.
	\item For the mapping torus $\pm M_A$ of a positive hyperbolic matrix $A\in \SL(2,\Z)$ acting on $T^2$ with either of the contact structures $\xi_{\pm}$ in \eqref{eq:maptor}, the ECH spectrum is infinite. 
	\item For the Brieskorn sphere $-\Sigma(2,3,6n+5)$, $n\geq 1$, with respect to the contact structure $\eta_\mathrm{tan}$ tangential to the Seifert fibration considered in \cite{Bow}, the ECH spectrum is infinite.
\end{enumerate}
\end{corollary}

\begin{remark}
	The embedded contact homology of prequantization bundles was computed by Farris \cite{Far} and Nelson--Weiler \cite{NW}. For the standard contact structure $\xi_{\can}$ on cotangent circle bundles, ECH can be related to string topology \cite{CL}. For the contact structures \eqref{eq:maptor} on the mapping torus $M_A$, ECH should be computable using toric techniques as in the work Choi \cite{Cho} and must coincide with the computation of Lebow \cite{Leb}. However, in none of these cases has the $U$ map been fully understood, and this is required to determine (in)finiteness of the ECH spectrum. One might be able to recover some of \Cref{cor:inf_spec} by careful study of the contact invariant and its relation with the $\F_2[U]$-module structure in Heegaard Floer or monopole Floer theories. Nevertheless, it is desirable to have an understanding of these objects purely in terms of contact dynamics and geometry. 
\end{remark}

One can obtain Liouville domains with more complicated topology by symplectic surgery procedures. In particular, one can attach Weinstein 1-handles between multiple of these Anosov Liouville domains (topologically this amounts to a boundary connected sum) to produce Liouville domains with disconnected boundary having arbitrarily many boundary components.

We should also mention a couple analogous constructions in higher dimensions. Geiges extended his earlier construction to six dimensions by finding suitable compact quotients of solvable Lie groups of dimension six \cite{Gei6}. Examples of a similar form were later found in all higher dimensions by Massot--Niederkr\"{u}ger--Wendl \cite{MNW} using some algebraic number theory.

\section{Proofs of the main theorems}
\label{sec:proof}

Let us now explain the proofs of our main results. 

\subsection{Proof of \Cref{thm:inf_spec_invts}}
We begin with a proof of \Cref{thm:inf_spec_invts}. In fact, using Hutchings \cite{HutTQFT} extension of ECH cobordism maps to strong cobordisms, we prove a more general statement. This extension will be used in the proofs of \Cref{thm:AnoReeb,thm:AnoHam,cor:Umap}.

\begin{theorem}
\label{thm:strong_inf}
Suppose $(X,\w)$ is a 4-dimensional connected strong symplectic cobordism with disconnected convex boundary 
\[(Y_+,\xi_+) = (Y_1,\xi_1) \sqcup \cdots \sqcup (Y_n,\xi_n) \qq{where} n>1.\] Further suppose the concave boundary $(Y_-,\xi_-)$ of $(X,\w)$ has non-vanishing contact invariant $c(Y_-,\xi_-)\neq 0$.	Then for each $1\leq i\leq n$, the ECH spectrum of $Y_i$ is infinite for any contact form $\lambda_i$ for $\xi_i$. Hence the ECH capacities of any weakly exact filling of $(Y_i,\xi_i)$ are also infinite. 
\end{theorem}

\begin{proof}

	Suppose $(X,\w)$ is a four-dimensional connected strong symplectic cobordism from 
	\[(Y_+, \lambda_+) = (Y_1,\lambda_1)\sqcup\cdots \sqcup(Y_n,\lambda_n) \qq{to} (Y_-,\lambda_-)\quad \text{with $n>1$.}\]
	Let $\xi_+=\ker\lambda_+$, $\xi_-=\ker\lambda_-$, and $\xi_i = \ker \lambda_i$ for $1\leq i\leq n$. We can assume that the cobordism has non-degenerate boundary after a perturbation of the ends.
	
	 Suppose for contradiction that the first element of the ECH spectrum of $(Y_1,\lambda_1)$ is finite. In particular, say $c_1(Y_1,\lambda_1)=L<\infty$. 	For $R>0$, we may enlarge our cobordism $X$ by gluing on a finite cylindrical end to form a strong cobordism
	\[ (X_R,\w_R) := (X,\w)  \bigcup_{\pd X =\{0\}\times Y_2} \Big([0,R]_t \times Y_2, \dd(e^t\lambda_2)\Big)\]
	which will have 
	convex boundary  
	\[(Y^R_+, \lambda_+) = (Y_1,\lambda_1)\sqcup (Y_2,e^R\lambda_1)\sqcup\cdots \sqcup(Y_n,\lambda_n).\]
	Take $R$ large enough so that one can find an irrational ellipsoid $E= E(2L, 2L+\varepsilon)$ symplectically embedded in $X_R$ (for example take a small ball contained in the cylindrical end of $Y_2$, perturb it, and enlarge under the Liouville flow). Then $E$ has non-degenerate contact boundary $(\pd E, \lambda_\std)$ with its shortest Reeb orbit of period $2L$.
	
	By definition of the ECH spectral invariants, one can find an element $\gamma \in ECH^L(Y_1,\lambda_1)$ so that $U_1^L\gamma = c(Y_1,\xi_1)$, where $U_1$ is the $U$ map on $Y_1$.  Then 
	 \begin{equation}
	 \label{eq:Ucomp}
	 \begin{split}
	 	U_1^L\Big(\gamma\otimes c(Y_2,\xi_2)\otimes\cdots\otimes  c(Y_n,\xi_n)\Big)&= (U_1^L\gamma )\otimes c(Y_2,\xi_2)\otimes\cdots \otimes  c(Y_n,\xi_n)\\
	 	 & = c(Y_+,\xi_+),
	 	 \end{split}
	 \end{equation}
	where we also write $U_1$ for the $U$ map on $Y_+^R$ with respect to a basepoint in $Y_1$.

	  The complement of $E$ inside $X_R$ is a strong symplectic cobordism which induces a cobordism map
	\[\Phi^L(X_R\setminus E,\w, 0): ECH^L(Y^R_+, \lambda_+,0) \to ECH^L(Y_- \sqcup \pd E, \lambda_-\sqcup \lambda_\std,0).\]

	Let $U_E$ denote the $U$ map on $\pd E$. By compatibility of the $U$ maps with cobordisms, one has
	\begin{align*}
		U_E^L \circ \Phi^L(X\setminus E, \w, 0)&(\gamma\otimes c(Y_2,\xi_2)\otimes\cdots\otimes  c(Y_n,\xi_n))\\
		 &= \Phi^L(X\setminus E, \w, 0)\circ U_1^L(\gamma\otimes c(Y_2,\xi_2)\otimes\cdots\otimes  c(Y_n,\xi_n)).
		 \intertext{Applying \eqref{eq:Ucomp},}
		&= \Phi^L(X\setminus E,\w,0) \,c(Y_+,\xi_+)\\
		&= c(Y_-,\xi_-)\otimes c(\pd E, \xi_\std).
	\end{align*}
	By assumption that $(Y_-,\xi_-)$ has non-vanishing contact invariant, this is a non-zero element in $ECH(Y_-\sqcup \pd E,\lambda_-\sqcup \lambda_\std)$. Further, the above equality shows that
	\[c(Y_-,\xi_-)\otimes c(\pd E,\xi_\std) \in \im\bigg(U_E^L: ECH^L(Y_-\sqcup \pd E,\lambda_-\sqcup \lambda_\std)\to ECH^L(Y_-\sqcup \pd E,\lambda_-\sqcup \lambda_\std)\bigg).\]
	Equivalently, 
	\[c(\pd E,\xi_\std) \in \im \bigg(U^L_E: ECH^L(\pd E,\lambda_\std) \to ECH^L(\pd E,\lambda_\std)\bigg).\]
	But the shortest Reeb orbit of $\pd E$ has action $2L$. So $ECH^L(\pd E,\lambda_\std)$ is isomorphic to $\F_2$, with its only non-trivial element being $c(\pd E,\xi_\std)$, and a trivial $U$-action. We thus obtain a contradiction. 
	
	Nothing was special about our choice of the boundary component $(Y_1,\lambda_1)$, and so we conclude $c_1(Y_i,\lambda_i)=\infty$ for $1\leq i\leq n$. Finiteness of the ECH spectrum depends only on the underlying contact structure, so the ECH spectrum of $(Y_i,\xi_i)$ will then be infinite for any choice of contact form. And given a weakly exact filling $(X,\w)$ of $(Y_i,\lambda)$, where $\ker\lambda=\xi_i$, we have
$c_1^\ECH(X,\w) = c_1(Y_i,\lambda)=\infty$. 
	\end{proof}

\begin{figure}
	\includegraphics[scale=0.5]{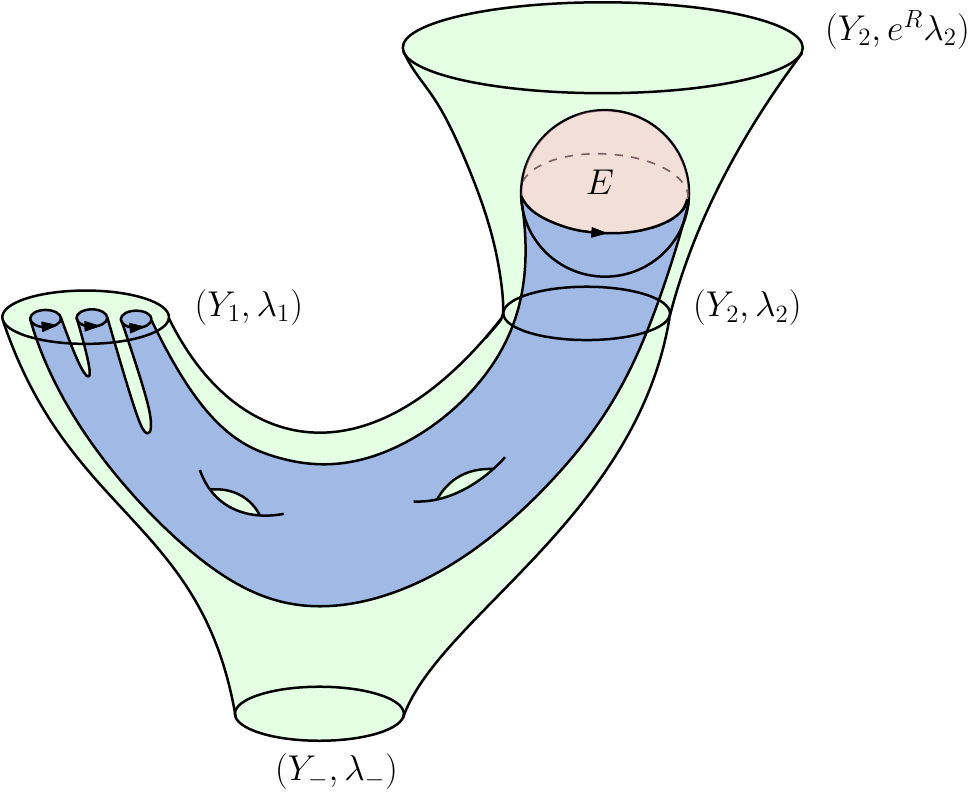}
	\caption{An illustration of the proof of \Cref{thm:strong_inf}. If $c_1(Y_1,\lambda_1)<\infty$, compatibility of the $U$ maps and cobordism map implies the existence of a holomorphic curve which does not respect the action filtration.}
	\label{img:cob}
\end{figure}

An illustration of the proof idea is given in \Cref{img:cob}. The figure depicts a holomorphic curve witnessed by the cobordism map $\Phi(X\setminus E, \w, 0)$ sending an ECH generator $\gamma$ of action $L$ to a Reeb orbit on the ellipsoid of action $2L$, which is incompatible with the action filtration.

Observe that \Cref{thm:inf_spec_invts} follows as a special case of \Cref{thm:strong_inf} where $Y_-$ is the empty manifold and the cobordism is weakly exact.  \Cref{thm:inf_spec_invts} also claims that if $(X,\w)$ is a weak Liouville domain with disconnected boundary $(Y_1,\lambda_1)\sqcup\cdots \sqcup (Y_n,\lambda_n)$, then its ECH capacities are infinite. To see this, note that the property \eqref{eq:specdisj} implies
\[c_k^\ECH(X,\w) = c_k\Big((Y_1,\lambda_1)\sqcup\cdots \sqcup (Y_n,\lambda_n)\Big) \geq  c_k(Y_1,\lambda_1).\]
We know $c_k(Y_1,\lambda_1)$ is infinite by \Cref{thm:strong_inf}, so the claim follows.

We note that while the proof of \Cref{thm:strong_inf} uses Hutchings' unpublished work \cite{HutTQFT}, \Cref{thm:inf_spec_invts} requires only that ECH cobordism maps are well defined for weakly exact cobordisms, which has appeared in the literature \cite{HTCob}.

\subsection{Proofs of \Cref{thm:AnoReeb,thm:AnoHam}}
We now address \Cref{thm:AnoReeb} using as input \Cref{thm:strong_inf}. The proof is adapted from some arguments of Barthelm\'{e}--Mann--Bowden \cite[Prop.\!\! A.1]{BarM}. We are grateful to Rohil Prasad for pointing this out in the literature.

\begin{proof}[Proof of \Cref{thm:AnoReeb}]
Suppose $(Y,\xi)$ is an Anosov contact structure with a contact form $\lambda$ and associated Reeb vector field $R$ whose Reeb flow is oriented Anosov. As in \Cref{ssec:AL}, consider an adapted metric $g$ for which the Anosov splitting
\[TY = \langle R\rangle \oplus  E^{ss}\oplus E^{uu}\]
is orthogonal. Let $e_s$ and  $e_u$ be oriented unit vector fields tangent to $E^{ss}$ and $E^{uu}$ and let $\alpha_s$ and $\alpha_u$ be their $g$-dual 1-forms. In particular, $R,e_s,e_u$ is an orthogonal frame for $Y$. Then as in \eqref{eq:pair},
\[\iota_R \dd\alpha_s =r_s \alpha_s \qand \iota_R \dd\alpha_u = r_u\alpha_u \quad \text{for $r_s<0<r_u$.}\]
 The $1$-form $\alpha_- = \alpha_u+\alpha_s$ satisfies
\[(\alpha_-\wedge\dd\alpha_-)(R,e_s,e_u)= r_s-r_u<0.\]
Hence 
\[\xi_- := \ker \alpha_ - = \langle R \rangle \oplus  \langle e_u-e_s\rangle \]
defines a negative contact structure on $Y$, or equivalently a positive contact structure on $-Y$. For $\varepsilon>0$, we can perturb this hyperplane distribution to
\[ \tilde{\xi}_-=\langle R + \varepsilon(e_s+e_u), e_u-e_s\rangle. \]
Because being contact is an open condition,  $\tilde{\xi}_-$  will still define a contact structure for $\varepsilon$ sufficiently small, and will now be transverse to the vector field $R$. Let $\eta_-$ be an arbitrarily $C^1$-close smooth approximation to this distribution. This will be a smooth contact structure transverse to $R$.  

Consider the truncated symplectization 
\[(X,\w) = \left([0,1]_t \times Y, \dd(e^t\lambda)\right).\]
Since $\lambda$ is a contact form for $\xi$, we have $\dd\lambda|_{\xi}>0$. Further, 
\[\dd\lambda(e_s,e_u) = \lambda \wedge \dd\lambda(R,e_s,e_u) >0.\]
Hence
\[\dd\lambda(R+\varepsilon(e_u+ e_s), e_u-e_s) = 2\varepsilon\cdot \dd\lambda(e_s, e_u)>0.\]
This means that $\dd\lambda|_{\tilde{\xi}_-}>0$ and consequently $\dd\lambda|_{\eta_-}>0$ for a close enough approximation. We deduce that $(X,\w)$ is a weak symplectic filling of $(Y,\xi)\sqcup (-Y, \eta_-)$. 

Because $\w$ is exact, the arguments of \cite[Prop.\!\ 4.1]{Eli} imply $\w$ can be modified near the boundary to a symplectic form $\tilde{\w}$ so that $(X, \tilde{\w})$ will be a strong filling of $(Y,\xi)\sqcup (-Y, \eta_-)$. Now \Cref{thm:strong_inf} implies that the ECH spectrum of $(Y,\xi)$ is infinite for any  contact form for $\xi$ and the ECH capacities of any exact filling of $(Y,\xi)$ are infinite.

On the other hand, we can also find a strong symplectic cap \cite{EH} for $(-Y, \eta_-)$ to obtain a strong filling of $(Y,\xi)$. And hence the contact invariant $c(Y,\xi)$ cannot vanish.
\end{proof}

We now prove \Cref{thm:AnoHam} using a generalization of the preceding argument. It is first helpful to introduce some standard terminology.

\begin{definition}
	Let $Y$ be a closed oriented 3-manifold. A \emph{Hamiltonian structure} on $Y$ is a non-vanishing closed 2-form $\w$. A \emph{framed Hamiltonian structure} on $Y$ is a pair $(\lambda, \w)$ of a 1-form $\lambda$ and a closed 2-form $\w$ so that $\lambda \wedge \omega >0$. To this pair, we associate a \emph{Hamiltonian vector field} $X$ uniquely specified by the conditions
	\[\iota_X\w =0 \qand \iota_X \lambda =1.\]
\end{definition} 
\begin{example}\label{ex:Hams} We recall the standard examples of framed Hamiltonian structures in three dimensions.
\begin{enumerate}[label=\textbf{\arabic*.}, leftmargin = 0.8cm]
	\item A contact form $\lambda$ determines a framed Hamiltonian structure $(\lambda,\dd\lambda)$. The associated Hamiltonian vector field is the usual Reeb vector field of $\lambda$.
	\item Let $(\Sigma, \sigma)$ be a closed surface with area-form and let $\phi: \Sigma \to \Sigma$ be an area-preserving diffeomorphism. We may form the associated mapping-torus of $\phi$ given as $Y= I_t \times \Sigma/(1,x)\sim (0,\phi(x))$. Then $(\dd t, \sigma)$ is a framed Hamiltonian structure on $Y$. The Hamiltonian vector field is $\pd_t$, generating the suspension flow.
	
	\item Given a symplectic 4-manifold $(X,\Omega)$ and a closed oriented hypersurface $i: Y \hookrightarrow X$ given as a regular level set $Y=H^{-1}(0)$ of a Hamiltonian $H: X\to \R$,  the 2-form $\w = i^\ast \Omega$ is a Hamiltonian structure on $Y$. Recall we can associate to $H$ a vector field $X_H$ on $X$ determined by $\dd H = \iota_{X_H}\Omega$, which will be tangent to $Y$. We can always find a 1-form $\lambda$ on $Y$ so that $\lambda(X_H|_Y)=1$.  Then $(\lambda,\w)$ will be framed Hamiltonian with Hamiltonian vector field $X_H|_Y$. \label{ex:Hamfield}
	
	\item Let $Y$ be a closed 3-manifold with a volume form $V$. Let $X$ be any non-vanishing volume-preserving vector field on $Y$. Then $\w=\iota_X V$ defines a Hamiltonian structure. One can find a 1-form $\lambda$ so that $\lambda (X)=1$, and consequently $(\lambda,\w)$ is framed Hamiltonian with Hamiltonian vector field $X$. 
\end{enumerate}
\end{example}

As in \Cref{ex:Hamfield} of  \Cref{ex:Hams}, consider a Hamiltonian hypersurface $Y=H^{-1}(0)$ in $(X,\Omega)$ equipped with the framed Hamiltonian structure $(\lambda,\w)$. A Moser-type argument for coisotropic submanifolds implies a tubular neighbourhood $U$ of $Y$ is symplectomorphic to $(-\varepsilon,\varepsilon)_t\times Y$ with symplectic form
\[\Omega_U =\w+\dd (t\lambda).\]
Moreover this symplectomorphism $\psi: (U,\Omega_U)\to ((-\varepsilon,\varepsilon)_t\times Y,\Omega)$ restricts to the identity map $\psi|_Y: Y\to \{0\}\times Y$.

\begin{proof}[Proof of \Cref{thm:AnoHam}]
	
	Suppose $M=H^{-1}(c)$ is a compact regular level set of a Hamiltonian $H:X\to \R$  in a weak Liouville domain $(X,\Omega)$ with $c_1^{\ECH}(X,\Omega)<\infty$ so that the characteristic flow on $M$ is oriented Anosov. That is, the Hamiltonian vector field $v = X_H|_M$ generates an oriented Anosov flow on $M$. We can also describe $v$ as the Hamiltonian vector field of a framed Hamiltonian structure $(\lambda,\w)$ on $M$. There is then a tubular neighbourhood $M\subset U$ symplectomorphic to $(-\varepsilon,\varepsilon)_t \times M$ with symplectic form $\Omega_U= \w + \dd(t\lambda)$. Up to multiplying $H$ by a sign (and changing the orientation and co-orientation of $M$ correspondingly), we  may assume the vector field $\pd_t$ along $M$ points out towards the boundary of $X$ (note the boundary is connected by \Cref{thm:inf_spec_invts}).
	
	 As in the previous proof of \Cref{thm:AnoReeb}, we consider an adapted metric $g$ for which the Anosov splitting
\[TM = \langle v\rangle \oplus  E^{ss}\oplus E^{uu}\]
is orthogonal. Let $e_s$ and  $e_u$ be oriented unit vector fields tangent to $E^{ss}$ and $E^{uu}$ and let $\alpha_s$ and $\alpha_u$ be their $g$-dual 1-forms. Then the $1$-form $\alpha_- = \alpha_u+\alpha_s$ satisfies $\alpha_-\wedge\dd\alpha_-<0$. Hence 
\[\xi_- := \ker \alpha_ - = \langle v \rangle \oplus  \langle e_u-e_s\rangle \]
defines a negative contact structure on $M$, or equivalently a positive contact structure on $-M$. For $\delta>0$ and sufficiently small, we can perturb this hyperplane distribution to another contact structure
\[ \tilde{\xi}_-=\langle v + \delta(e_s+e_u), e_u-e_s\rangle \]
transverse to $v$ and let $\eta_-$ be an arbitrarily $C^1$-close smooth approximation to $\tilde{\xi}_-$.  

For $\varepsilon>0$ and sufficiently small, we have a symplectic manifold 
\[(W,\Omega_W) = \Big( [-\varepsilon,0]_t \times M,  \w + \dd(t \lambda)\Big).\]
We see that
\[(\w - \varepsilon\dd\lambda)(R+\delta(e_u+ e_s), e_u-e_s) = 2\delta\cdot \w(e_s, e_u)  -\varepsilon\cdot \dd\lambda(R, e_u-e_s)-2\varepsilon\delta \cdot \dd\lambda(e_s,e_u).\]
Note $\lambda\wedge\w(X, e_s, e_u)>0$ so that $\w(e_s,e_u)>0$. We may then take $\varepsilon$ sufficiently small relative to $\delta$ so that $w-\varepsilon \dd\lambda$ restricted to $\tilde{\xi}_-$ is positive, and thus positive on $\eta_-$ as well. This implies $\{-\varepsilon\}\times -M$ is a weakly convex boundary component of $(W,\Omega_W)$ with respect to the contact structure $\eta_-$.

Let $(D,\Omega|_D) \subset (X,\Omega)$ be the subdomain bounded by $M$. Then $X\setminus D$ can be glued to $W$ along $\{0\}\times M$, since both manifolds are modelled on $(U,\Omega_U)$ at their boundary. Let $(Y,\xi)$ be the contact-type boundary of $(X,\Omega)$.

After gluing, we obtain a symplectic structure $\tilde{\Omega}$ on $X\setminus D$ with strongly convex boundary $(Y,\xi)$ and weakly convex boundary $(\{-\varepsilon\}\times -M, \eta_-)$. 

Because $X$ is a weak Liouville domain, $\Omega$ is exact and hence $\w$, which is its pullback to $M$, is also exact. Thus $(X\setminus D, \tilde{\Omega})$ is exact near the boundary component $\{-\varepsilon\}\times -M$ and  \cite[Prop.\! 4.1]{Eli} implies $\tilde{\Omega}$ can then be modified so that both boundary components of $X\setminus D$ are strongly convex. We have thus produced a strong co-filling of $(Y,\xi)$. On the other hand, by assumption on the first ECH capacity of $(X,\Omega)$, we have $c_1(Y,\xi)<\infty$, which contradicts \Cref{thm:strong_inf}.\end{proof}

\begin{remark} 
The same proof implies the following more general result. Suppose $(X,\w)$ is a strong symplectic cobordism from $(Y_+,\xi_+)$ to $(Y_-,\xi_-)$, with $c_1(Y_+,\xi_+)<\infty$, and the contact invariant $c(Y_-,\xi_-)\neq 0$. Given a separating \emph{exact} Hamiltonian hypersurface $Y\subset X$, its characteristic flow cannot be oriented Anosov. Here exact means that the associated Hamiltonian structure $\omega|_Y$ is an exact 2-form.
\end{remark}

\begin{corollary}\label{cor:elliptic}
	Suppose that $(X,\w)$ is a four-dimensional weak Liouville domain so that $c_1^{\ECH}(X,\w)<\infty$. Let $H: X\to \R$ be a smooth Hamiltonian with a compact regular level set $M=H^{-1}(c)$ so that $M$ is either:
	\begin{itemize}
		\item an integer homology sphere,
		\item or homeomorphic to a circle bundle over an oriented surface.
	\end{itemize} 
	Then in some $C^2$-open neighbourhood of $H$, a $C^2$-dense open set of Hamiltonians $\widetilde{H}$ have an elliptic orbit on $\widetilde{H}^{-1}(c)$.
\end{corollary}
\begin{proof}
As Newhouse \cite{New} shows, in some open neighbourhood $\mathcal{U} \subset C^\infty(M)$ of $H$ with the $C^2$-topology, there is an open dense set $\mathcal{V}\subset \mathcal{U}$ so that for $\widetilde{H}\in \mathcal{V}$, the flow on $\widetilde{H}^{-1}(c)$ either is Anosov, or has an elliptic orbit. To obtain our result, it will thus suffice to show the flow on $\widetilde{H}^{-1}(c)$ cannot be  Anosov.

By Ehresmann's theorem, we can suppose that for $\widetilde{H}\in \mathcal{U}$, the hypersurface $\widetilde{H}^{-1}(c)$ is diffeomorphic to $M= H^{-1}(c)$. We claim that under our assumptions on $M$, all Anosov flows on $M$ are oriented. Then \Cref{thm:AnoHam} rules out $\widetilde{H}^{-1}(c) \subset X$ carrying an oriented Hamiltonian Anosov flow, from which the result follows.

If $M$ is homeomorphic to a circle bundle over an oriented surface $\Sigma_g$, then Ghys \cite{Ghy} shows any Anosov flow on $M$ is orbit equivalent to a fibrewise pullback of the geodesic flow on the cotangent circle bundle $S^\ast\Sigma_g$. Since the geodesic Anosov flow is oriented, its pullback to $M$ must be oriented as well. 

Given any Anosov flow on $M$, the obstruction to its orientability is the Stiefel--Whitney class $w_1(E^u)=w_1(E^s) \in H^1(M;\Z/2\Z)$ of the weak distributions of the flow. If $M$ is an integer homology sphere (or just a $\Z/2\Z$-homology sphere), the group $H^1(X;\Z/2\Z)$ is trivial. Hence any Anosov flow on $M$ will be automatically oriented.
\end{proof}

\begin{remark}
	We can readily find hypersurfaces which satisfy the assumptions of \Cref{cor:elliptic}. There are lots of integer homology spheres which embed as hypersurfaces in $\R^4$, including many Brieskorn spheres \cite{CH} (most of which admit an Anosov flow). One can attach two null-homologous one-handles to the zero-section of $D^\ast S^2$, to produce a genus 2-surface in $D^\ast S^2$. The boundary of a tubular neighbourhood of this surface will be homeomorphic to $S^\ast \Sigma_2$. 
\end{remark}

\subsection{Proof of \Cref{cor:Umap}}
	Consider any contact structure $(Y,\xi)$ on $Y$. Associated to the contact structure is its contact invariant $c(Y, \xi) \in ECH(Y,\xi,0)$. 	Assume $c(Y,\xi)=0$; then $(Y,\xi)$ is not strongly fillable \cite{HutTQFT, Ech}. If $(Y,\xi)$ were strongly co-fillable, one could find a strong cap \cite{EH} for the other ends to produce a strong filling of $(Y,\xi)$. Hence $(Y,\xi)$ is not strongly co-fillable.
	
	If instead $c(Y,\xi)\neq 0$, then one has $U[\emptyset]=0$ for action reasons and so $c(Y,\xi) \in \ker(U)$. Taubes \cite{Tau} shows there is an isomorphism of $\F_2[U]$-modules
	\[\widehat{HM}^{-\ast}(Y,\mathfrak{s}_\xi) \cong ECH_\ast(Y,\xi, 0).\]
	By the hypothesis that $\ker(U)\subset \im(U)$, we conclude there is some $\gamma \in ECH(Y,\xi,0)$ so that
	\[U(\gamma) = [\emptyset] = c(Y,\xi).\]
	But the first ECH spectral invariant $c_1(Y,\lambda)$ is then finite, in particular bounded by the minimal action needed to represent $\gamma$ by Reeb orbit sets for $(Y,\lambda)$.  \Cref{thm:strong_inf} then implies $(Y,\xi)$ is not strongly co-fillable.
	
	If $Y$ possessed an oriented Anosov flow, then the associated Anosov Liouville domain would give a strong co-filling $I\times Y$ of some contact structure $\xi$ on $Y$. Moreover, from its construction, this contact distribution has a non-vanishing section given by the vector field generating the Anosov flow. Thus the Chern class $c_1(\xi)=c_1(\mathfrak{s}_\xi)=0$. If we assume $\ker(U)\subset \im(U)$ for the spin$^c$-structure $\mathfrak{s}_\xi$, we obtain a contradiction with the existence of a strong co-filling.\qed

\section{Further constructions}
\label{sec:constructions}
We now prove some other results discussed in the introduction and consider some examples. We also present a few open questions surrounding ``elementary capacities."

\subsection{Hypertautness}\label{ssec:hyp}
We can generalize the results of \Cref{sec:proof} from the setting of Anosov flows to hypertaut foliations. We are grateful to Thomas Massoni and Audrey Rosevear for some comments inspiring these extensions.

\begin{definition}
	Let $\mathcal{F}$ be a co-oriented $C^0$-foliation on $Y$. We say that $\mathcal{F}$ is \emph{hypertaut} if there is a smooth one-form $\beta$ on $Y$ so that $\dd\beta|_{T\mathcal{F}}>0$. This is equivalent to the existence of an exact volume preserving flow transverse to $\mathcal{F}$.
\end{definition}
As explained by Massoni \cite{Mas3}, a hypertaut foliation gives rise to a pair of approximating contact structures via \cite{ET} with a Liouville co-filling.

\begin{proposition}[{\cite[Prop.\!\! 4.4]{Mas3}}]\label{prop:hyp}
	Suppose $\mathcal{F}$ is a hypertaut $C^0$-foliation on $Y$ defined by $T\mathcal{F}=\ker\alpha$ and so that $\dd\beta|_{T\mathcal{F}}>0$. Then $\mathcal{F}$ is $C^0$-approximated by a pair of positive/negative contact structures $\xi_\pm$ so that $\dd\beta|_{\xi_\pm}>0$. Moreover $\alpha$ can be smoothed to $\widetilde{\alpha}$ satisfying $\widetilde{\alpha}\wedge\dd\beta>0$ so that for $\varepsilon>0$ sufficiently small, $\lambda=\beta+t\widetilde{\alpha}$ is the primitive of a symplectic form on $[-\varepsilon,\varepsilon]_t\times Y$ with weakly convex boundary $(Y,\xi_+)\sqcup (-Y,\xi_-)$. After modifying $\lambda$ as in \cite{Eli}, we obtain an exact filling of $(Y,\xi_+)\sqcup (-Y,\xi_-)$.
\end{proposition}
This filling has less structure than the Liouville pairs\footnote{If one assumes $\mathcal{F}$ is $C^2$, then one does in fact get a Liouville pair \cite[Cor.\! 4.6]{Mas3}.} and Anosov--Liouville domains that Mitsumatsu constructs for Anosov flows, but it is sufficient to apply \Cref{thm:strong_inf}. With cosmetic modifications to their proofs, \Cref{thm:AnoReeb,thm:AnoHam,cor:Umap} have the following respective generalizations.

\begin{proposition}\label{prop:hypspec}
	Suppose $(Y,\xi)$ is a contact 3-manifold admitting a Reeb flow transverse to a $C^0$-foliation. Or suppose $(Y,\xi)$ belongs to a contact pair approximating a hypertaut $C^0$-foliation as in \Cref{prop:hyp}. Then the ECH spectrum of $(Y,\xi)$ is infinite.
\end{proposition}

\begin{proposition}
	Suppose $(X,\w)$ is a four-dimensional weak Liouville domain with $c_1^{\ECH}(X,\w)<\infty$. Let $M= H^{-1}(c)$ be a compact regular level set of a smooth Hamiltonian $H: X\to \R$. Then the characteristic flow on $M$ cannot be transverse to a co-orientable $C^0$-foliation of $M$.
	\end{proposition}
	
\begin{proposition}
	Let $Y$ be a closed oriented three-manifold and $a\in H^2(Y;\Z)$. Suppose that the $U$ map on monopole Floer cohomology
	\[U: \bigoplus_{\{\mathfrak{s}:\, c_1(\mathfrak{s})=a\}} \widehat{HM}^\ast(Y,\mathfrak{s}) \to \bigoplus_{\{\mathfrak{s}:\, c_1(\mathfrak{s})=a\}}\widehat{HM}^\ast (Y,\mathfrak{s})\]
	satisfies $\ker(U) \subset \im(U)$. Then $Y$ admits no co-orientable $C^0$-hypertaut foliation $\mathcal{F}$ with Euler class $e(\mathcal{F})=a$.
\end{proposition}

We present a sample application of \Cref{prop:hypspec} to periodic Reeb flows. 
\begin{definition}
	A contact 3-manifold $(Y,\xi)$ is called \emph{Besse} if it admits a Reeb flow with all of its orbits closed. In this case, the flow must be periodic and the orbits define the structure of a Seifert fibration on $Y$ \cite{CGM}. If all orbits have the same minimal period, then the flow is called \emph{Zoll} and $(Y,\xi)$ is a prequantization contact structure.
\end{definition}

A Seifert fibred 3-manifold will admit a Besse Reeb flow if and only if it has negative Euler class \cite[Prop.\! 3.1]{LM}. Work of Eisenbud--Hirsch--Neumann, Jankins, and Naimi establishes necessary and sufficient conditions in terms of the Seifert invariants for a Seifert fibration to admit a foliation transverse to the fibration (see \cite[Thm.\!\ 1.2]{LM}). Such a transverse foliation is sometimes called \emph{horizontal}. In the case of a bonafide circle bundle $Y\to \Sigma_g$, this condition simplifies to the Milnor--Wood inequality; namely for $g\geq 1$, there is a horizontal foliation if and only if $|e(Y)| \leq 2g-2$. When both a Besse Reeb flow and a horizontal foliation exist, \Cref{prop:hypspec} implies the following.

\begin{corollary}\label{cor:Besse}
	Let $(Y,\xi)$ be a Besse contact 3-manifold.
	\begin{enumerate}[label=\textbf{\arabic*.}, leftmargin = 0.8cm, itemsep=0.2cm]
		\item If the Seifert-fibration defined by the Besse Reeb flow satisfies any of the conditions of \cite[Thm.\!\ 1.2]{LM}, then the ECH spectrum of $(Y,\xi)$ is infinite.
		\item In particular, if $(Y,\xi_\pre)$ is a prequantization contact structure over a genus $g\geq 2$ base and $e(Y) \geq 2-2g$, then the ECH spectrum of $(Y,\xi)$ is infinite.
	\end{enumerate}
\end{corollary} 
\begin{remark}
	The extremal case of prequantization bundles $(Y,\xi)$ with $e(Y)=2-2g$ was covered by \Cref{cor:inf_spec}. The ECH spectrum of some prequantization bundles over $S^2$ is known to be finite \cite{FR}, and we suspect more generally it is finite whenever \Cref{cor:Besse} doesn't apply.
\end{remark}
\begin{remark}
	Work of Cristofaro-Gardiner--Mazzucchelli \cite[Thm.\! 3.2]{CGM} characterizes Besse Reeb flows in terms of ECH spectral invariants. They prove that a contact 3-manifold $(Y,\lambda)$  is Besse if and only if there is a class $\gamma \in ECH(Y,\lambda)$ with $U\gamma\neq 0$ so that the ECH spectral invariants associated to $\gamma$ and $U\gamma$ coincide:
	\[\inf\{L\in\R| \gamma \in ECH^L(Y,\lambda)\}=\inf\{L\in\R | U\gamma \in ECH^L(Y,\lambda)\}.\]
	Our result implies in general it is not sufficient to check equality of elements of the ECH spectrum, one must look at the full collection of ECH spectral invariants.
\end{remark}

\subsection{Non-orientable surfaces and ECH capacities}
\label{ssec:non-orient}
Recall that the closed non-orientable surfaces $N_k$ are indexed by a positive integer $k$, where $N_k$ is the connected sum of $k$ copies of $\R P^2$. We have $N_1=\R P^2$, $N_2$ is the Klein bottle $K$, $N_3$ is Dyck's surface, and so forth. The Euler characteristic of $N_k$ is $2-k$ and its orientable double cover is the oriented surface $\Sigma_{k-1}$ of genus $k-1$.  We now prove the following.
\begin{proposition}
	Let $N$ be a closed non-orientable surface equipped with a Riemannian metric $g$. Then the ECH capacities of the unit disk cotangent bundle $D_g^\ast N$ associated to $g$ are finite and satisfy a Weyl law:
	\[\lim_{k\to \infty} \frac{c_k^{\mathrm{ECH}}(D_g^\ast N, \w_\can)^2}{k} = 4\pi \mathrm{vol}(N, g).\]
\end{proposition}
\begin{proof}
	The ECH capacities of $D^\ast \R P^2$ for the round metric are computed by Ferreira--Ramos \cite{FR}. The capacities of $D^\ast K$ for the flat metric are computed by Miranda--Ramos \cite{MR}. For $j\geq 1$, Givental \cite{Giv} constructed Lagrangian embeddings of the non-orientable surfaces $N_{2+4j}$ with Euler characteristic divisible by four into $\R^4$. A suitably small tubular neighbourhood of these Lagrangians can be identified with a rescaling of $D^\ast N_{2+4j}$. The Lagrangian and its disk bundle are contained in some sufficiently large ball $B^4(R) \subset \R^4$, and so finiteness follows from monotonicity of the ECH capacities.
	
	Given a Lagrangian $L$ in a symplectic 4-manifold $(X,\w)$, we can perform a symplectic blow-up along a symplectic ball $B^4(\lambda)$  passing through $L$. Through a careful construction of Rieser \cite{Rie}, we may perform the blow-up so that $L$ lifts to a Lagrangian $\tilde{L}$ in the blow-up $(X\# \overline{\C P^2}(\lambda) ,\widetilde{\w}_\lambda)$ which smoothly is the real blow-up $L \# \R P^2$.
	
	Combined with the construction of Givental, we conclude that for every non-orientable surface $N_k$ for $k\geq 6$, there is a symplectic embedding of $D^\ast N_k$ into $B^4(R)$ blown-up in at most three points. Note the blow-ups only alter $B^4(R)$ in the interior, so that the contact boundary of the blown-up ball is still the standard tight contact 3-sphere $(S^3, \xi_{\std})$. Hence the above embeddings gives rise to strong symplectic cobordisms $(X_k,\w_k)$ from $(S^3, \lambda_\std)$ to $(S^\ast N_k , \varepsilon\lambda_\can)$ for some $\varepsilon$ sufficiently small. For each $i >0$, let $\gamma_i \in ECH(S^3, \lambda_\std,0)$ satisfy $U^i \gamma_i = [\emptyset]$. Then
	\[U^i \circ \Phi(X_k,\w_k, 0)\gamma_i = \Phi(X_k,\w_k,0)\circ U^i \gamma_i = \Phi(X_k,\w_k,0)[\emptyset]=[\emptyset]. \]
	Hence $[\emptyset] \in \im (U^i : ECH(S^\ast N_k,\xi_\can)\to ECH(S^\ast N_k,\xi_\can))$ for each $i$ and the ECH capacities of $(D^\ast N_k, \w_\can)$ are finite for $k\geq 6$.

	It remains to consider $N_3, N_4$, and $N_5$. For these examples, begin with the disk cotangent bundle of the Klein bottle $D^\ast K$. By the same Lagrangian surgery under blow-up, this time in a neighbourhood of a point in the zero-section, we can find a symplectic embedding of $D^\ast N_{2+j}$ into the blow-up of $D^\ast K$ at $j$ many points for $j=1,2,3$. The complement of the embeddings give rise to strong symplectic cobordisms from $(S^\ast K, \xi_\can)$ to $(S^\ast N_{2+k}, \xi_\can)$. The same argument as above shows the ECH capacities of $D^\ast N_k$ are finite for $k=3,4,5$. 
	
	The Weyl law follows immediately from the more general ECH Weyl law of Cristofaro-Gardiner--Hutchings--Ramos \cite{Weyl}, observing that $\mathrm{vol}(D_g^\ast N, \w_\can) =  \pi \mathrm{vol}(N,g)$, since each disk fibre has area $\pi$. 
\end{proof}

\begin{remark}
	It would be interesting to quantitatively control the ECH capacities of $D^\ast N_k$ for $k\geq 3$. For example, for a choice of hyperbolic metric $g$ on $N_k$, one may try to bound $c_1^{\ECH}(D^\ast_g N_k,\w_\can)$ from above in terms of some multiple of the length of the systole of $(N_k,g)$.
	\end{remark}
\begin{remark}
One would like to also understand the subleading asymptotics of the ECH Weyl law. One expects the capacities to remain a bounded distance from their leading asymptotic term, i.e.
\begin{equation}\label{eq:subleading}
	 c_k^\ECH(D^\ast_g N, \w_\can) = \sqrt{4\pi \mathrm{vol}(N,g)k} +O(1) \quad (k\to\infty).
\end{equation}
Recent remarkable work of Edtmair \cite{Edt} implies that if $(X,\w)$ is symplectic with smooth boundary and embeds into a symplectic manifold for which the error term in the Weyl law of the ECH capacities is bounded, then the same is true $(X,\w)$. In particular, this implies the Weyl law error term is $O(1)$ for all smooth star-shaped domains.

It is easy to check the error term in the ECH spectrum Weyl law is $O(1)$ for the standard contact form on $S^3$ and the contact form on $S^\ast \R P^2$ given by the round metric, as in \cite{FR}. By Edtmair's work, this then implies the formula \eqref{eq:subleading} is satisfied for any Finsler metric $g$ where $N=\R P^2$ or $N=N_{2+4k}$ is a non-orientable surface with a Lagrangian embedding in $\R^4$. 

By a version of Polterovich's Lagrangian surgery, which is described in the next subsection, if a surface $L$ has a Lagrangian embedding into $(X,\w)$, then $L\#4\R P^2$ does too. This implies there is a Lagrangian embedding of $N_{1+4k}$ into $T^\ast \R P^2$ for all $k$. Hence \eqref{eq:subleading} is satisfied whenever $N=N_{1+4k}$ as well. By the same surgery argument, it would suffice to show \eqref{eq:subleading} for $N=N_2,N_3,N_4$ to prove it for all remaining cases.
\end{remark}

\subsection{High-genus Lagrangians and infinite capacities}
\label{ssec:Lags}

In light of \Cref{cor:Lags}, one way to demonstrate the ECH capacities of a symplectic 4-manifold are infinite is to find within it an embedded Lagrangian submanifold of genus two or greater. We now present a couple of pertinent examples.

Our constructions will use Polterovich's Lagrangian surgery procedure \cite{Pol}. Suppose $L_1, L_2$ are two Lagrangians in a symplectic 4-manifold which intersect transversely at a point $p$. In a Darboux chart around $p$, we see two Lagrangian planes in $\R^4$ intersecting at the origin. Polterovich shows we may locally smooth out this intersection to a tube $I\times S^1$ connecting the two Lagrangians. We will thus obtain a Lagrangian which is topologically the connected sum $L_1\# L_2$. The same procedure applies to self-intersections, but with an additional subtlety regarding orientations. If $L$ is an oriented Lagrangian with a transverse self-intersection at $p$, the intersection index at $p$ may be plus or minus one. If it is plus one, surgery will give $L\# T^2$. If the index is minus one, the tube we glue in swaps orientations and we obtain a connected sum with the Klein bottle $L \# K$. If $L$ is non-orientable, the surgery can either be described as $L\# T^2$ or $L \# K$ (topologically these are the same).

\begin{enumerate}[label=\textbf{\arabic*.}, leftmargin = 0.8cm]
\item Consider $T^4$ with coordinates $\theta_1,\ldots, \theta_4$ and a product symplectic form 
\[\w_{a,b}=a\cdot  \dd\theta_1\wedge \dd\theta_2+b\cdot \dd\theta_3\wedge\dd\theta_4\qq{with} a,b>0.\] 
Given two embedded Lagrangian tori  $ \{x_1\} \times S^1 \times \{x_1\} \times S^1$ and $S^1 \times \{x_2\} \times S^1 \times \{x_2\}$ inside $T^4$,  Polterovich's surgery produces a genus two Lagrangian in $(T^4,\w_{a,b})$.  We can apply the same surgery procedure to $\{x_1\} \times S^1 \times \{x_1\} \times S^1$ and a collection of Lagrangians 
\[S^1 \times \{x_2\} \times S^1 \times \{x_2\},\ldots , S^1 \times \{x_g\} \times S^1 \times \{x_g\}\]
to produce a Lagrangian of any higher genus $g$. 

More generally, let $(\Sigma, \sigma)$ and $(\Pi,\rho)$ be  surfaces equipped with area forms and each of genus at least one. Consider two closed curves $\gamma_1, \eta_1 \subset \Sigma$ intersecting transversely at a single point, and likewise for $\gamma_2, \eta_2 \subset\Pi$. Then $\gamma_1 \times \gamma_2$ and $\eta_1\times \eta_2$ are Lagrangian tori in $\Sigma \times \Pi$ with the product symplectic form which transversely intersect at a single point. Surgery yields a Lagrangian of genus two. 

\item One might also consider a non-trivial symplectic fibration 
\[(\Pi,\rho) \hookrightarrow (M,\w) \to (\Sigma,\sigma)\] with monodromy determined by a homomorphism $\phi: \pi_1(\Sigma) \to \mathrm{Symp}(\Pi, \rho)$. If one can find curves $\gamma_1, \eta_1 \subset \Sigma$ and $\gamma_2, \eta_2 \subset\Pi$ as above so that $\phi([\gamma_1])$ preserves $\gamma_2$ and $\phi([\eta_1])$ preserves $\eta_2$, then we can repeat the above construction to yield a genus two Lagrangian in $(M,\w)$.

As one example of this, recall the Kodaira--Thurston manifold, which was the first construction of a symplectic manifold with no K\"{a}hler structure \cite{Thu}. This  is the $T^2$ bundle over $T^2$ with monodromy $\phi: \pi_1(T^2) \cong \Z^2 \to \mathrm{SL}(2,\Z)$ given by
\[\phi(1,0) = \pmqty{1 & 0\\0& 1} \qand \phi(0,1) = \pmqty{1&1\\0&1}.\] Let $\gamma_1$ and $\eta_1$ denote simple closed curves in $T^2$ which intersect transversely at a single point and represent $(1,0)$ and $(0,1)$ in $\pi_1 (T^2)$ respectively. Let $\eta_2$ be a closed curve in $T^2$ preserved by the Dehn twist $\phi(0,1)$ and let $\gamma_2$ be any closed curve intersecting $\eta_2$ once transversely. Then as above, $\gamma_1\times \gamma_2$ and $\eta_1\times \eta_2$ are Lagrangian tori which can be surgered to yield a genus two Lagrangian.

\item Consider a projective K3 surface $X$ with a K\"{a}hler form $\w$, and suppose $\sigma: X\to X$ is an anti-symplectic involution so that $\sigma^\ast\w =-\w$. The fixed points of $\sigma$ then define a Lagrangian surface $X^\sigma$  in $X$. For example, if $X$ is the vanishing set of a quartic in $\C P^3$, then the real locus $X_\R = X \cap \R P^3$ will be one such Lagrangian. It turns out $X^\sigma$ is either a pair of tori, or a surface of genus up to ten along with a collection of spheres \cite{Nik}. In particular, we may find projective K\"ahler K3 surfaces with Lagrangians of genus $g\leq 10$. A particularly classical example goes as follows. Let $C\to \C P^2$ be a smooth sextic curve, which then has genus 10. Let $r: X\to \C P^2$ be a ramified degree two cover along $C$. Then this is a K3 with an anti-symplectic covering involution $\sigma$. The fixed point set $X^\sigma = r^{-1}(C)$ is a genus 10 Lagrangian. 

An alternative construction of Lagrangians of any genus in symplectic K3 surfaces is given by Sivek--Van~Horn-Morris \cite{SVHM} using the description of the elliptic surface $E(2)$ as a symplectic K3. Recall $E(1):=\C P^2\# 9\overline{\C P^2}$ admits a singular torus fibration over the sphere with twelve singular fibres. If $\alpha$ and $\beta$ are simple closed curves in $T^2$ intersecting once transversely, the Picard--Lefschetz formula implies we may take $E(1)$ so that half of the singular fibres have vanishing cycle $\alpha$ and the other half have vanishing cycle $\beta$. Then $E(2):=E(1)\#_F E(1)$ is formed from the Gompf fibre sum of two copies of $E(1)$ along regular fibres $F$. A neighbourhood of the sum region is identified with $T^2 \times (0,1)\times S^1$ in which one can find $g$ disjoint Lagrangian tori $\alpha \times \{\mathrm{pt}\} \times S^1$. We can also find a Lagrangian sphere living over a ``matching path" between singular fibres with vanishing cycle $\beta$ in the two copies of $E(1)$ \cite[\S 16]{Sei}. This sphere will intersect each torus once transversely; the resulting surgery is a genus $g$ Lagrangian in $E(2)$.  A similar idea works for other elliptic surfaces 
\[E(n,g) = \left(\scalebox{1.3}{\#}_F^n E(1)\right)\#_F (T^2\times \Sigma_g)\]
by applying further fibre sums to $E(2)$. 

\item A \emph{plumbing} of two cotangent disk bundles $D^\ast M_1 \#_{p} D^\ast M_2$ is given by choosing balls $B_1 \subset M_1$, $B_2 \subset M_2$ and gluing the cotangent bundles via a symplectomorphism $D^\ast B_1 \cong D^\ast B_2$ that interchanges the base and fibre directions. The plumbing has the structure of a Liouville domain with the zero sections $M_1$ and $M_2$ as embedded Lagrangians. These Lagrangians intersect transversely at a point and so we may also resolve this intersection to obtain a Lagrangian copy of the connected sum $M_1 \# M_2$. In this way one can plumb along a tree of surfaces and produce a Lagrangian of genus $g\geq 2$, provided the sum of the genera of the surfaces in the plumbing is at least two.
\end{enumerate}

\begin{corollary}\label{cor:Lag_egs}
The following hold. 
\begin{enumerate}[label=\textbf{\arabic*.}, leftmargin = 0.8cm]
	\item The ECH capacities of any product of closed oriented surfaces each of genus at least one are infinite.
	\item The ECH capacities of the Kodaira--Thurston manifold \cite{Thu} are infinite for any form arising from a symplectic fibration.
	\item On every elliptic surface $E(n,g)$ with $n\geq 2$, there are symplectic forms for which the ECH capacities are infinite.
	\end{enumerate}
	\begin{enumerate}[resume,label=\textbf{\arabic*.}, leftmargin = 0.8cm]
	\item For any plumbing of cotangent disk bundles over a tree of surfaces of combined genus at least two, all its ECH capacities are infinite.
\end{enumerate}
\end{corollary}

\subsection{Constructing Anosov Hamiltonian hypersurfaces}
\label{ssec:AnoHamconstruction}

We give a very simple construction of Anosov Hamiltonian hypersurfaces in $\R^4$, addressing Herman's \Cref{q:Herm} and proving \Cref{prop:AnoHamconstruction}.

Note to prove \Cref{prop:AnoHamconstruction}, it suffices to construct infinitely many pairwise non-homeomorphic Hamiltonian hypersurfaces in $\R^4$, since by Darboux's theorem, these will then give hypersurfaces in an arbitrary symplectic 4-manifold.

Recall that there are non-orientable closed embedded Lagrangian submanifolds in $\R^4$ of any negative Euler characteristic divisible by four, as first shown by Givental \cite{Giv}. Let $L_n \subset \R^4$ be any such Lagrangian with Euler characteristic $\chi(L_n) = -4n$. Topologically, $L_n$ is the connected sum of $4n+2$ copies of $\R P^2$. Because it has negative Euler characteristic, $L_n$ admits a hyperbolic metric $g_n$. Let $D^\ast L_n$ denote the unit-disk bundle of $L_n$ with respect to the metric $g_n$. By the Weinstein neighbourhood theorem, for some sufficiently small $\varepsilon$, a tubular neighbourhood of $L_n$ gives a symplectic embedding $i: (D^\ast L_n, \varepsilon \w_\can) \hookrightarrow (\R^4,\w_\std)$. Let $h_n: D^\ast L_n\to \R$ be the norm-squared function induced by the metric $g_n$ on the unit tangent bundle, i.e. for $(x,v)\in D^\ast L_n$,  let $h_n(x,v)= \|v\|^2_{g_n}$. Take $H_n: \R^4 \to \R$ to be any smooth function which agrees with $h_n$ on the image $i(D^\ast L_n)$ and is greater than or equal to one away from $i(D^\ast L_n)$.

  Consider $Y_n = H_n^{-1}(1/2)$, which is a (contact-type) hypersurface in $\R^4$. This is a cotangent circle bundle over $L_n$ and the characteristic flow of $X_{H_n}$ on $Y_n$ is the cotangent lift of geodesic flow on $L_n$ with respect to $g_n$ (slowed down by a factor of $\frac{\varepsilon}{2}$). Because $g_n$ is hyperbolic, this flow is (non-oriented) Anosov. Clearly the hypersurfaces $Y_n$ for distinct $n$ are pairwise non-homeomorphic.

\begin{remark}
Note for $n\geq 3$, there cannot be a closed embedded (orientable or non-orientable) Lagrangian admitting a metric of negative sectional curvature in $\R^{2n}$ \cite[Thm.\!\! 1.7.5]{EGH}. In fact, because Anosov geodesics are always non-contractible with Morse index zero \cite{Kli},	 the same SFT-based proof works for the (potentially larger) class of Lagrangians admitting an Anosov geodesic flow. Hence Anosov geodesic flow can never appear on a Hamiltonian hypersurface in $\R^{2n}$ with $n\geq 3$. 

Further, a suspension flow (Anosov or otherwise) also cannot appear on a Hamiltonian hypersurface $Y$ in $\R^{2n}$. To observe this, such a flow would mean $Y$ fibres over $S^1$ and a fibre gives a closed embedded $(2n-2)$-manifold $Z \subset Y$ positively transverse to the suspension flow. If the suspension flow were generated by a Hamiltonian structure $\Omega$, then $\int_Z \Omega^{n-1} >0$. But if $Y$ is a Hamiltonian hypersurface in $\R^{2n}$, then the Hamiltonian structure $\Omega$ is pulled back from the symplectic form $\w_{\std}$ on $\R^{2n}$. In particular, $\Omega$ is exact and Stokes' theorem implies $\int_Z \Omega^{n-1} =0$.

Examples of Anosov flows in dimensions greater than three which are not of geodesic or suspension type are rather scarce. Hence it appears probable Herman's \Cref{q:Herm} has a negative answer in higher dimensions.
\end{remark}

\subsection{Speculations on elementary capacities}
There are a collection of capacities
\[0<c_1^{\mathrm{Alt}}(X,\w) \leq c_2^{\mathrm{Alt}}(X,\w)\leq c_3^{\mathrm{Alt}}(X,\w)\leq \cdots \leq +\infty\] 
introduced by Hutchings \cite{HutAlt} as an elementary alternative to the ECH capacities. These behave similarly to the ECH capacities but are defined in terms of max-min problems over the abstract space of $J$-holomorphic curves and avoid the use of ECH and Seiberg--Witten theory. Roughly, the capacity $c_k^{\mathrm{Alt}}(X,\w)$ measures the minimal $L>0$ so that for every (admissible in the appropriate sense) almost complex structure $J$ and collection of $k$ points, there is a $J$-holomorphic curve through these points with energy at most $L$. These capacities are inspired by similar invariants defined by McDuff--Siegel \cite{MS} as an elementary alternative to Siegel's higher capacities \cite{Sie}.

For Liouville domains, one always has the inequality
\[c_k^{\mathrm{Alt}}(X,\w) \leq c_k^{\mathrm{ECH}}(X,\w).\]
Thus in the cases above where we have shown the ECH capacities to be infinite, it is undetermined whether the alternative capacities are infinite. 
\begin{question}\label{q:alt}
	Given a four-dimensional Liouville domain $(X,\w)$ with disconnected boundary, is it guaranteed that its alternative ECH capacities $c_k^\mathrm{Alt}(X,\w)$ are infinite? 
\end{question}

Recall that, unlike with the ECH capacities, the finiteness of alternative ECH capacities is not an intrinsic property of the boundary (although no examples of this phenomenon are yet known as far as the author is aware). The curves counted by alternative capacities may depend heavily on one's choice of filling. Thus the following  question is independent (in either direction) of the previous one.

\begin{question}
	Given a closed oriented surface $\Sigma_g$ of genus $g\geq 2$, are the alternative ECH capacities of a compact cotangent disk bundle $c_k^\mathrm{Alt}(D^\ast\Sigma_g,\w)$ infinite?
\end{question}

Note that in the case of genus $g=0$, it is known that $c_k^\mathrm{Alt}(D^\ast S^2, \w)$ is strictly less than $c_k^\mathrm{ECH}(D^\ast S^2, \w)$ for certain $k$ \cite[Rmk.\!\!\! 14]{HutAlt}. So it is plausible that in higher genus the alternative ECH capacities see holomorphic curves missed by ECH.

Given \Cref{prop:rsft}, we may also be interested in how capacities constructed from (rational) SFT act in domains with disconnected boundary. These include the capacities $\mathfrak{g}_k(X,\w)$ and $\mathfrak{r}_k(X,\w)$ constructed by Siegel \cite{Sie}. Also there are the elementary versions $\widetilde{\mathfrak{g}}_k(X,\w)$ defined by McDuff--Siegel \cite{MS}. 

\begin{question}\label{q:sie}
	Given a Liouville domain $(X,\w)$ with disconnected boundary, is it guaranteed that Siegel's capacities $\mathfrak{g}_k(X,\w)$ or $\mathfrak{r}_k(X,\w)$ are infinite? Is it guaranteed that its McDuff--Siegel capacities $\widetilde{\mathfrak{g}}_k(X,\w)$ are infinite? 
\end{question}

\begin{figure}[h]
	\includegraphics[scale=0.65]{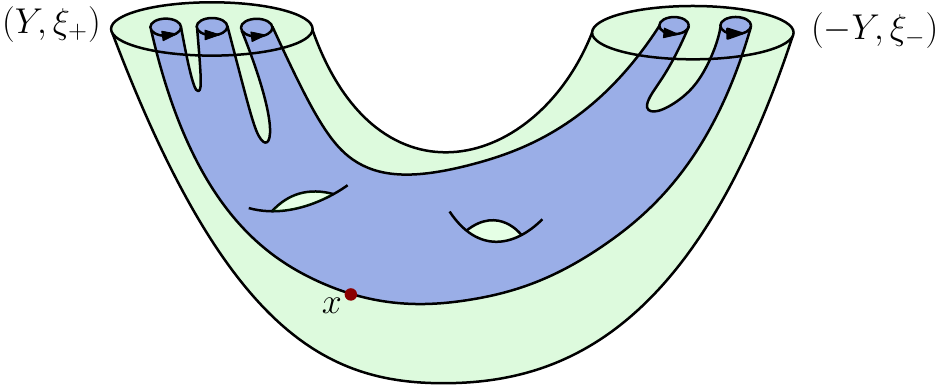}
	\caption{A point-constrained curve that might plausibly exist in an Anosov Liouville domain $I\times Y$ and be seen by $c_1^{\mathrm{Alt}}(I\times Y,\w)$.}
	\label{img:cob_alt}
\end{figure}

Note that the content of the proof of \Cref{thm:inf_spec_invts} is not sufficient to resolve these questions in the affirmative. And in fact, the author's best guess is that the answers to Questions \ref{q:alt} and \ref{q:sie} are no, at least for some of the Anosov Liouville domains. Roughly speaking, in proving \Cref{thm:inf_spec_invts} we rule out non-zero algebraic counts of holomorphic curves in $(X,\w)$ satisfying a point constraint with positive ends all in a single boundary component of $X$. But for example, this does not rule out the existence of point-constrained holomorphic curves with positive ends at \emph{all} of the boundary components of $X$ (see \Cref{img:cob_alt}). Such curves are not seen by the ECH or RSFT capacities we consider, but they may exist and their energy might provide a finite bound on the alternative ECH capacities or the other capacities of Siegel and McDuff--Siegel. 

As a first step towards answering \Cref{q:alt}, one might try to understand the ECH cobordism map associated to an Anosov Liouville domain.

\appendix
\section{Analogues for rational symplectic field theory}
\label{app:rsft}
The proof of \Cref{thm:inf_spec_invts} is general enough to apply to some other SFT and Floer-type spectral invariants, including in higher dimensions. For example, an analogous theorem will hold for the Gutt--Hutchings capacities \cite{GH} of Liouville domains with disconnected boundary. However, this is not so interesting because, as noted in \Cref{rmk:taut}, the Liouville domains $I\times Y$ we consider have hypertight contact forms on their boundary (including the examples in higher dimensions).

 Heuristically, the first Gutt--Hutchings capacity of a Liouville domain $(X,\w)$ measures the action of a Reeb orbit which bounds a (perturbed) index two pseudoholomorphic plane. In particular, the orbit must be contractible in $X$. But if $(Y,\lambda)$ is hypertight, then all its Reeb orbits cannot be contractible in the domain $I\times Y$. Hence the first and all higher Gutt--Hutchings capacities of $I\times Y$ will be infinite if its boundary is hypertight. 
 
 \begin{remark}
 	It is an interesting question if there is a Liouville domain of the form $I\times Y$ whose boundary components do not admit hypertight contact forms. Note one can find more complicated Liouville domains with disconnected non-hypertight boundary. For example, one can take an Anosov Liouville domain and attach a pair of Liouville domains with virtually overtwisted boundary to the two boundary components via Weinstein one-handles. The result is a Liouville domain with two virtually overtwisted (and hence not hypertight) boundary components. Of course the Gutt--Hutchings capacities of this must be infinite because it contains a Liouville embedding of a subdomain with hypertight boundary.
 \end{remark}
 
If we instead consider theories which count non-cylindrical curves, then understanding their spectral invariants becomes more complicated. We study one such example of invariants constructed from a ``$q$-only" formulation of rational symplectic field theory, as sketched by Hutchings \cite{HutRSFT}.

\subsection{Brief review of rational SFT}
Symplectic field theory (SFT) was originally proposed by Eliashberg--Givental--Hofer \cite{EGH}. Wendl \cite{WenSFT} has written an excellent manuscript describing much of the background for and details of SFT. We follow a version of rational SFT explained by Hutchings \cite{HutRSFT}; a more algebraically-sophisticated version of this theory has been explored by Latschev, Siegel, and Moreno--Zhou \cite{Lat, Sie, MZ, MZ2}.

\begin{assumption}\label{assump:science_fiction}
	We are assuming from here on out that one has developed a suitable  perturbation framework, say from polyfolds, implicit atlases, or global Kuranishi charts, to coherently define virtual oriented counts of rational holomorphic curves in symplectic field theory when transversality fails. These counts should be coherent for different variants of SFT and agree with counts of honest holomorphic curves when transversality can be obtained. 
\end{assumption}

Let $(Y,\lambda)$ be a contact manifold of any odd dimension $2n-1$; i.e. $\lambda$ is a one-form on $Y$ and $\lambda\wedge (\dd\lambda)^{n-1}>0$. The Reeb vector field and its orbits are defined as in dimension three. We assume that $\lambda$ is generic so that all its Reeb orbits are \emph{non-degenerate}, again defined as in dimension three. To an orbit $\gamma$, we can associate a \emph{Conley--Zehnder index} $\mu_{\mathrm{CZ}}(\gamma)$, which is independent of choices modulo two (see \cite[\S3.4]{WenSFT} for details). We say that a Reeb orbit $\gamma$ is \emph{bad} if it is a double cover of another Reeb orbit whose Conley--Zehnder index is of opposite parity, and otherwise we say $\gamma$ is \emph{good}.  

To every good orbit $\gamma$, we associate a variable $q_\gamma$. Let $\mathcal{A}$ be the ideal of non-constant polynomials inside the supercommutative $\Q$-algebra on variables $q_\gamma$ for all good Reeb orbits $\gamma$; we give this a $\Z/2$-grading by declaring
\[|q_\gamma|= n+1 +\mu_{\mathrm{CZ}}(\gamma).\]
Let $\mathfrak{A}$ be the supercommutative symmetric algebra on $\mathcal{A}$. 

We also define a $\Z/2$-graded differential algebra $\mathfrak{D}$ which is the $\Q$-algebra generated by the variables $q_\gamma$ and formal derivatives $\pdv{q_\gamma}$, each of grading $n+1+\mu_{\mathrm{CZ}}(\gamma)$, for every good orbit $\gamma$. We make $\mathfrak{D}$ a supercommutative Poisson algebra by imposing for each $\gamma$ that
\[\{\pdv{q_{\gamma}}, q_\gamma\} = 1 \]
 and setting all other pairs of variables to Poisson supercommute.
 There is a natural action
\[\mathfrak{D}\times \mathfrak{A}\to \mathfrak{A},\quad (\mathbf{d}, x) \mapsto \mathbf{d}\cdot x\]
where an operator $q_{\alpha_1}\cdots q_{\alpha_k} \pd_{q_{\beta_1}} \cdots \pd_{q_{\beta_\ell}}$ acts as follows on a symmetric product $x_1\otimes \cdots \otimes x_n $ of elements of $\mathcal{A}$. We can let each partial derivative $\pd_{\beta_i}$ act on a distinct term $x_j$ by a graded Leibniz rule and form the product in $\mathcal{A}$. Multiplying that on the left by $q_{\alpha_1}\cdots q_{\alpha_k}$ yields a new element of $\mathcal{A}$. The tensor product of this with all the terms $x_t$ which did not have a derivative applied yields an element of $\mathfrak{A}$. The action of $q_{\alpha_1}\cdots q_{\alpha_k} \pd_{q_{\beta_1}} \cdots \pd_{q_{\beta_\ell}}$ is the sum over all such elements of $\mathfrak{A}$ from all choices of which derivative $\pd_{\beta_i}$ acts on which term $x_j$. This action is compatible with the Poisson structure in the sense that for homogeneous elements $\mathbf{d},\mathbf{e}\in\mathfrak{D}$ and $x\in\mathfrak{A}$,
\begin{equation}\label{eq:Pois}
\{\mathbf{d}, \mathbf{e}\} \cdot x = \mathbf{d}\cdot (\mathbf{e}\cdot x) - (-1)^{|\mathbf{d}|\cdot |\mathbf{e}|}\mathbf{e}\cdot (\mathbf{d}\cdot x).	
\end{equation}

 Choose a generic admissible almost-complex structure $J$ on $\R\times Y$. We then define a differential operator called the \emph{RSFT Hamiltonian}
\[ \mathbf{h} := \sum_{\Gamma_+,\Gamma_-,A} \frac{\#^{\mathrm{vir}}\left(\mathcal{M}_{0,0}^{\mathrm{ind}\,1}(\R\times Y, J, A, \Gamma_+, \Gamma_-) \right)}{|\Gamma_+|!\,|\Gamma_-|!\,\kappa_{\Gamma_-}} q^{\Gamma_-} \pdv{q^{\Gamma_+}} \in \mathfrak{D}.\]
Here we are summing over all lists of good orbits $\Gamma_+, \Gamma_-$  and relative homology classes $A\in H_2(Y, \Gamma_+\cup \Gamma_-)$. The term $\#^{\mathrm{vir}}(\cdots) \in \Q$ is the orbifold (virtual) count of genus-zero index-one connected finite-energy $J$-holomorphic curves in $\R\times Y$ positively/negatively asymptotic to $\Gamma_\pm$ in the homology class $A$ modulo reparametrization and their automorphisms (which makes sense in light of \Cref{assump:science_fiction}). Also $\kappa_{\Gamma_+}$ is the product of the covering multiplicities of all orbits in $\Gamma_+$ and $q^{\Gamma_\pm}$ is the product of the $q^\gamma$ terms for all $\gamma$ in $\Gamma_\pm$.

Under \Cref{assump:science_fiction}, the combinatorics of the boundaries of the moduli spaces of index 2 genus zero curves imposes the relation $\{\mathbf{h}, \mathbf{h}\}=0$. One checks from the index formula for holomorphic curves that $\mathbf{h}$ has odd degree, so  \eqref{eq:Pois} implies for all $x\in \mathfrak{A}$ that
\[\mathbf{h}\cdot (\mathbf{h}\cdot x) = \frac{1}{2} \{\mathbf{h},\mathbf{h}\}\cdot x =0. \]
We thus obtain a chain complex $(\mathfrak{A}, \pd)$ where $\pd x = \mathbf{h}\cdot x$.
\begin{definition}
We define \emph{($q$-only) rational symplectic field theory} $RH(Y,\lambda)$ to be the homology of $(\mathfrak{A}, \pd)$.
\end{definition}

Given a disjoint union $(Y,\lambda)=(Y_1, \lambda_1)\sqcup \cdots \sqcup (Y_k,\lambda_k)$, there is a natural map
\begin{equation}\label{eq:rsftdis}
 j: RH(Y_1,\lambda_1)\otimes\cdots \otimes RH(Y_k, \lambda_k) \to RH(Y,\lambda)
\end{equation}
given on the chain level as the inclusion of algebras $\mathfrak{A}_{Y_1}\otimes\cdots \otimes \mathfrak{A}_{Y_k} \hookrightarrow \mathfrak{A}_Y$. This map is injective; it has a right inverse which on the chain level takes a symmetric product of algebra generators and divides it into the symmetric product of generators from each component $Y_i$.

Rational SFT has an $H_1(Y;\Z)$-grading given by
\[RH(Y,\lambda) = \bigoplus_{\Gamma\in H_1(Y)} RH(Y,\lambda, \Gamma),\]
where $RH(Y,\lambda, \Gamma)$ is the homology of the submodule $\mathfrak{A}(\Gamma) \subset \mathfrak{A}$ generated by elements whose total homology class is $\Gamma$. This makes sense because for any generator $x$ of $\mathfrak{A}$, $\mathbf{h}\cdot x$ is always homologous to $x$ and so $(\mathfrak{A}(\Gamma), \pd)$ is a subcomplex.

\subsection{Structures on rational SFT}

As in ECH, there is an $\R$-filtration on $\mathfrak{A}$. For $L>0$, let $\mathfrak{A}^L$ be the subspace of $\mathfrak{A}$ generated by symmetric products of monomials in the variables $q_\gamma$ so that the total action of all variables in the symmetric product is less than $L$. By Stokes' theorem, every non-zero term in the operator $\mathbf{h}$ has $\mathcal{A}(\Gamma_+)\geq \mathcal{A}(\Gamma_-)$ and so the differential $\pd$ preserves this filtration. Hence rational symplectic field theory obtains an $\R$-filtration and we may define $RH^L(Y,\lambda)$ as the homology of the subcomplex $(\mathfrak{A}^L,\pd)$.

Given an exact symplectic cobordism $(X,\w)$ from $(Y_+,\lambda_+)$ to $(Y_-,\lambda_-)$, there is a filtered homomorphism
\[\Phi^L(X,\w): RH^L(Y_+,\lambda_+)\to RH^L(Y_-,\lambda_-) \quad \text{for $L>0$}\]
which is induced on the chain level from a map $\phi^L(X,\w,J)$ counting rational $J$-holomorphic curves in the symplectic completion of $X$. The precise algebraic formalism is explained in \cite{HutRSFT}. These maps are functorial under composition of cobordisms. In particular this shows $RH^L(Y,\lambda)$ is independent of the almost-complex structure and the unfiltered homology $RH(Y,\lambda)$ depends only on the contact structure $\xi = \ker \lambda$. When $(X,\w)$ is only a weakly exact cobordism, we still obtain a filtered homomorphism
\[\Phi^L(X,\w): RH^L(Y_+,\lambda_+,0)\to RH^L(Y_-,\lambda_-,0)\]
on the $0\in H_1(Y)$ graded piece of rational SFT. 

For $(X,\w)$ a strong symplectic cobordism, there are additional complications arising from the possible presence of closed curves in the cobordism. However, as with ECH, we can restrict to counting curves representing the zero class in $H_2(X, \pd X)$ to define a filtered homomorphism
\[\Phi^L(X,\w,0): RH^L(Y_+,\lambda_+,0)\to RH^L(Y_-,\lambda_-,0).\]

We can also define an analogue of the $U$ map on $RH(Y,\lambda)$. Pick a generic point $y\in Y$. We form a distinguished element of the differential algebra $\mathfrak{D}$ given by
\[\mathbf{u}_y := \sum_{\Gamma_+,\Gamma_-,A} \frac{\#^{\mathrm{vir}}\left(\mathcal{M}_{0,1}^{\mathrm{ind}\,2n-2}(\R\times Y, J, A, \Gamma_+, \Gamma_-)\cap \mathrm{ev}^{-1}\{(0,y)\} \right)}{|\Gamma_+|!\,|\Gamma_-|!\,\kappa_{\Gamma_-}} q^{\Gamma_-} \pdv{q^{\Gamma_+}}. \]
The term $\#^{\mathrm{vir}}(\cdots)$ is the rational orbifold (virtual) count of genus-zero index $(2n-2)$ connected finite-energy $J$-holomorphic curves in $\R\times Y$ positively/negatively asymptotic to $\Gamma_\pm$ in the homology class $A$  with a single marked point mapping to $(0,y)\in \R\times Y$ modulo automorphisms.

Keeping track of how curves break in the moduli space of index $2n-1$ curves through $(0,y)$ implies that $\mathbf{h}$ and $\mathbf{u}_y$ Poisson-commute. As a consequence, $\mathbf{u}_y$ is a chain map with an induced map on homology
\[U_y^L: RH^L(Y,\lambda) \to RH^L(Y,\lambda) \quad \text{ for each $L>0$.}\]
A generic path between points $y, y' \in Y$ determines a chain homotopy between $U_y$ and $U_{y'}$, hence there is a single canonical map on homology when $Y$ is connected. For disconnected $(Y,\lambda)= (Y_1,\lambda_1)\sqcup \cdots \sqcup (Y_k,\lambda_k)$ and for a generic point $y\in Y_1$, under the injective map \eqref{eq:rsftdis} we obtain
\begin{equation}\label{eq:rsftucom}
U_{y} \circ j = j \circ U_{Y_1},	
\end{equation}
where $U_y$ is the $U$ map on $RH(Y,\lambda)$ associated to the component $Y_1$ and $U_{Y_1}$ is the $U$ map on $RH(Y_1,\lambda_1)$.
 
Given a connected (weakly) exact cobordism $(X,\w)$ from $(Y_+,\lambda_+)$ to $(Y_-, \lambda_-)$ and a pair of points $y_+\in Y_+$ and $y_- \in Y_-$, studying curves in the completion of $X$ with a point constraint along a path between $y_+$ and $y_-$ gives that
\begin{equation}\label{eq:rsftcob}
\Phi^L(X,\w) \circ U^L_{y_+} = U^L_{y_-} \circ \Phi^L(X,\w) \quad \text{for any $L>0$.}	
\end{equation}
The same holds for the map $\Phi^L(X,\w,0)$ when $(X,\w)$ is only a strong cobordism.

\subsection{RSFT capacities}
We can now associate a sequence of capacities to a Liouville domain $(X,\w)$ using this version of rational symplectic field theory:
\[0< c^{\mathrm{RSFT}}_1(X, \w) \leq c_2^{\mathrm{RSFT}}(X,\w) \leq  c_3^{\mathrm{RSFT}}(X,\w)\leq \cdots \leq \infty.\]
These are defined in close analogy with the ECH capacities.

Let $(Y,\lambda)$ be a non-degenerate contact manifold. First note there is a distinguished element in $\mathfrak{A}$ which is simply the unit of the symmetric algebra; we denote this by $\emptyset$. By Stokes' theorem, this is a cycle and so defines a (possibly trivial) class $[\emptyset] \in RH(Y,\lambda)$. 
\begin{lemma}\label{lem:rsftempty}
	Given a weakly exact symplectic cobordism $(X,\w)$ between non-degenerate contact manifolds $(Y_1,\lambda_1)$ and $(Y_2,\lambda_2)$, the RSFT cobordism map respects the unit:
	\[\Phi(X,\w) [\emptyset]=[\emptyset].\]
	For $(X,\w)$ a strong cobordism, $\Phi(X,\w,0)$ respects the unit too.
\end{lemma}

\begin{proof}
	By Stokes' theorem, the chain-level cobordism map $\phi(X,\w,J)$ is action-decreasing. Thus $\phi(X,\w,J)$ must send the empty orbit set of $(Y_1,\lambda_1)$ to the empty orbit set of $(Y_2, \lambda_2)$ by counting the empty holomorphic curve. For a strong cobrodism, the same applies to the map $\Phi(X,\w,0)$, since we count null-homologous curves.
\end{proof}
Note as a corollary of \eqref{eq:rsftcob} and \Cref{lem:rsftempty}, the (canonical) isomorphism $RH(Y,\lambda) \cong RH(Y,\lambda')$ for different contact forms is compatible with the $U$ map and preserves the class $[\emptyset]$.

\begin{definition}
The \emph{RSFT spectrum} of a non-degenerate contact manifold $(Y,\lambda)$ is the sequence of spectral invariants
\[0< c_1^\mathrm{RSFT}(Y,\lambda) \leq c_2^{\mathrm{RSFT}}(Y,\lambda)\leq c_3^\mathrm{RSFT}(Y,\lambda)\leq \cdots \leq \infty\]	
given by
\[c_k^\mathrm{RSFT}(Y,\lambda) := \inf\left\{L>0\,\Big|\,[\emptyset] \in \im\Big(U^k: RH^L(Y,\lambda)\to RH^L(Y,\lambda)\Big) \right\}.\]
When $Y$ is disconnected, this means that $[\emptyset]$ should be in the image of any $k$-fold composition of $U$ maps associated to the different components of $Y$. For $(Y,\lambda)$  degenerate, we define $c_k^{\mathrm{RSFT}}(Y,\lambda)$ as the limit of the spectral invariants of any non-degenerate sequence $\{\lambda_n\}_{n=1}^\infty$ of contact forms $C^\infty$-converging to $\lambda$.

We then define the \emph{RSFT capacities} of any weakly exact filling $(X,\w)$ of $(Y,\lambda)$ by
\[c_k^\mathrm{RSFT}(X,\w) := c_k^\mathrm{RSFT}(Y,\lambda).\]
\end{definition}

\begin{lemma}\label{lem:rsftcapprops}
The RSFT capacities satisfy the following properties.	
\begin{enumerate}[label=\textbf{\arabic*.}, leftmargin = 0.8cm]
		\item Under a symplectic embedding $(X,\w) \hookrightarrow (X',\w')$, the capacities are non-decreasing: 
		 \[c_k^{\mathrm{RSFT}}(X,\w) \leq c_k^{\mathrm{RSFT}}(X',\w')\quad  \text{for all $k$}.\]
		\item The capacities scale linearly with symplectic area: for any $s>0$,
		\[c_k^{\mathrm{RSFT}}(X,s\w) = s\cdot c_k^{\mathrm{RSFT}}(X,\w) \quad \text{for all k.}\]
		\item The capacities satisfy a disjoint union property: 
				\[c_k^{\mathrm{RSFT}}\left(\coprod_{i=1}^n(X_i,\w_i)\right) = \max_{k_1+\cdots+k_n=k} \sum_{i=1}^n c_{k_i}^{\mathrm{RSFT}}(X_i,\w_i).\]
	\end{enumerate}
\end{lemma}
\begin{proof} Let $(X,\w)$ and $(X',\w')$ have contact boundaries $(Y,\lambda)$ and $(Y',\lambda')$ respectively. By perturbations, we may assume the boundaries are non-degenerate.
\begin{enumerate}[label={\arabic*.}, leftmargin = 0.8cm]
	\item  Suppose $c_k^{\mathrm{RSFT}}(X',\w')< L$ so there is $\gamma \in RH^L(Y',\lambda')$ which satisfies $U^k\gamma=[\emptyset]$. Given a symplectic embedding $\psi: (X,\w)\hookrightarrow (X',\w')$, the complement of the image of $\psi$ defines a weakly exact  cobordism with an associated cobordism map
	\[\Phi^L(X'\setminus \psi(X),\w) : RH^L(Y',\lambda',0)\to RH^L(Y,\lambda,0).\]
	Then 
	\begin{align*}
		U^k\circ \Phi^L(X'\setminus \psi(X),\w)\gamma &= \Phi^L(X'\setminus \psi(X),\w)\circ U^k\gamma\\
		&=\Phi^L(X'\setminus \psi(X),\w)[\emptyset]\\
		&=[\emptyset].
	\end{align*}
	Hence $[\emptyset] \in \im(U^k)$ on $RH^L(Y,\lambda)$ and so $c_k^{\mathrm{RSFT}}(X,\w)\leq L$.
	\item This is immediate from the fact there is a canonical isomorphism
	 \[RH^L(Y,\lambda) \cong RH^{sL}(Y, s\lambda)\]
	 for $s>0$ and this isomorphism respects the element $[\emptyset]$ and intertwines the actions of $U$. One can obtain this isomorphism from the cobordism map associated to the truncated symplectization $([0,s]_t\times Y, \dd(e^t\lambda))$, which for a translation invariant almost-complex structure can only count trivial cylinder curves.
	 \item Let $(X_i,\w_i)$ have boundary $(Y_i,\lambda_i)$ and let $(Y,\lambda)$ denote the union of the components $(Y_i,\lambda_i)$. Consider the surjective map $j^\vee : \mathfrak{A}_Y \to \mathfrak{A}_{Y_1}\otimes \cdots\otimes \mathfrak{A}_{Y_n}$  which divides a symmetric product of generators into the pieces coming from the various connected components of $Y$. This is the chain level right inverse to the map $j$ from \eqref{eq:rsftdis}. Because the holomorphic curves counted by $\mathbf{h}$ are connected, $j^\vee$ is a chain map. For the same reason, $j^\vee$ will commute with the chain-level map $\mathbf{u}_y$ for $y$ in any connected component of $Y$. If one can find $x_i \in RH^{L_i}(Y_i,\lambda_i)$ so that $U_{Y_i}^{k_i}x_i = [\emptyset]$ for each $i$, then the relation \eqref{eq:rsftucom} implies 
	 \[U_{Y_n}^{k_n}\circ \cdots \circ U_{Y_1}^{k_1} \circ j(x_1\otimes\cdots \otimes x_n) = [\emptyset].\]  
	 Note $j$ is compatible with the action filtrations, in particular $j(x_1\otimes \cdots \otimes x_n) \in RH^{L_1+\cdots + L_n}(Y,\lambda)$. We conclude that
	 \[c_k^\mathrm{RSFT}(Y,\lambda) \leq \max_{k_1+\cdots+k_n=k} \sum_{i=1}^n c_{k_i}^{\mathrm{RSFT}}(Y_i,\lambda_i).\]
	 Applying the corresponding argument for $j^\vee$ will give 
	 \[c_k^\mathrm{RSFT}(Y,\lambda) \geq \max_{k_1+\cdots+k_n=k} \sum_{i=1}^n c_{k_i}^{\mathrm{RSFT}}(Y_i,\lambda_i).\]
\end{enumerate}
\end{proof}
We can extend our definition of the RSFT capacities to all other symplectic manifolds $(X,\w)$ by taking the capacities to be the supremum of the capacities of any weak Liouville domain which embeds into $(X,\w)$. \Cref{lem:rsftcapprops} then holds for all symplectic manifolds.
\begin{remark}
Siegel \cite{Sie} constructs a wide catalogue of symplectic capacities from rational symplectic field theory. In particular, he defines two sequences of capacities $\mathfrak{g}_1(X,\w), \mathfrak{g}_2(X,\w),\ldots$ and $\mathfrak{r}_1(X,\w),  \mathfrak{r}_2(X,\w),\ldots$ which measure the action of certain rational holomorphic curves satisfying tangency constraints and point constraints in $X$ respectively. The capacities $\mathfrak{r}_k(X,\w)$ are quite similar to our capacities $c_k^{\mathrm{RSFT}}(X,\w)$ but not the same in general. Our capacities measure rational holomorphic curves with point constraints that can be taken to live entirely in the symplectization of the boundary of $X$ (see \cite[Rmk.\!\!\!\! 5.3]{Sie}). Moreover, Siegel's capacities are defined in terms of certain spectral invariants coming from classes in the bar complex of linearized contact homology or the contact homology algebra, instead of $q$-only rational SFT.
\end{remark}

\begin{lemma}\label{lem:rsftcap}
Finiteness of the RSFT spectrum depends only on the contact structure. I.e. if $(Y,\lambda)$ is a contact manifold with underlying contact structure $\xi=\ker\lambda$ so that $c_k^{\mathrm{RSFT}}(Y,\lambda)=\infty$, then the same is true for every other  contact form for $\xi$.
\end{lemma}
\begin{proof}
	First assume $(Y,\lambda)$ is non-degenerate. By definition, finiteness of $c_k^{\mathrm{RSFT}}(Y,\lambda)$ is equivalent to $[\emptyset]$  lying in the image of $U^k$ on $RH(Y,\lambda)$. Since the $U$ map and the contact invariant depend only on $\xi=\ker \lambda$, the $k$th element of the RSFT spectrum is infinite for one non-degenerate contact form if and only if it is infinite for all.
	
	If $c_k^{\mathrm{RSFT}}(Y,\lambda)$ is infinite for all non-degenerate contact forms $\lambda$, continuity implies it will also be infinite for the degenerate contact forms. If $c_k^{\mathrm{RSFT}}(Y,\lambda)=\infty$ for some degenerate $\lambda$, then monotonicity implies $c_k^{\mathrm{RSFT}}(Y,e^{f}\lambda)=\infty$ for any non-negative function $f:Y\to \R_{\geq 0}$. For a generic such function, $e^f\lambda$ will be non-degenerate, implying $c_k^{\mathrm{RSFT}}$ is infinite for all contact forms.\end{proof}

\subsection{RSFT capacities for domains with disconnected boundary}
We now have a version of \Cref{thm:inf_spec_invts} for our new capacities.

\begin{proposition}\label{prop:rsft}
Suppose $(X,\w)$ is a  strong symplectic filling with convex boundary 
\[(Y_1,\lambda_1)\sqcup \cdots \sqcup (Y_n,\lambda_n)\qq{for}  n>1.\]
Then for $1\leq i \leq n$, the RSFT spectrum of $(Y_i,\lambda_i)$ is infinite. Hence the RSFT capacities of $(X,\w)$ and any weakly exact filling of $(Y_i,\ker\lambda_i)$ are infinite.
\end{proposition}
\begin{proof}
The proof is a minor modification of that of \Cref{thm:strong_inf}. 	Consider a Liouville domain $(X,\w)$ with convex boundary
	\[(Y, \lambda) = (Y_1,\lambda_1)\sqcup\cdots \sqcup(Y_n,\lambda_n) \quad \text{for $n>1$}.\]
	
	 We may approximate to assume $(Y,\lambda)$ is non-degenerate. Suppose for contradiction that the RSFT spectrum of $(Y_1,\lambda_1)$ is finite with $c_1(Y_1,\lambda_1)=L<\infty$. We can then find an element $\gamma \in RH^L(Y_1,\lambda_1)$ with $U\gamma = [\emptyset]$. 
	 
	 First we enlarge our cobordism $X$ by gluing on a cylindrical end to form
	\[ (X_R,\w_R) := (X,\w)  \bigcup_{\pd X=\{0\}\times Y_2} \Big([0,R]_t \times Y_2, \dd(e^t\lambda_2)\Big)\]
	with $R$ large enough so that one can find an irrational ellipsoid $E$ which symplectically embeds into $X_R$ whose shortest Reeb orbit has action $2L$. Then $X_R$ has convex boundary  
	\[(Y^R_+, \lambda_+) = (Y_1,\lambda_1)\sqcup (Y_2,e^R\lambda_1)\sqcup\cdots \sqcup(Y_n,\lambda_n).\]
 The complement of $E$ inside $X_R$ is a strong symplectic cobordism which induces a cobordism map
	\[\Phi^L(X_R\setminus E,\w,0): RH^L(Y^R_+, \lambda_+,0) \to RH^L(\pd E, \lambda_{\std},0).\]
	 Let $U_1$ denote the $U$ map on $RH(Y,\lambda)$ associated to a point in $Y_1$ and $U$ the $U$ map on $RH(Y_1,\lambda_1)$. In light of \eqref{eq:rsftucom} and the fact $j$ is induced from an inclusion of algebras, 
	 \begin{equation*}
	 	U_1^L\circ j(\gamma)= j\circ U_1^L\gamma= j([\emptyset])=[\emptyset].
	 \end{equation*}
	Let $U_E$ denote the $U$ map on $RH(\pd E,\lambda_\std)$. In light of \eqref{eq:rsftcob} and \Cref{lem:rsftempty},
	\begin{align*}
		U_E^L \circ \Phi^L(X\setminus E, \w,0)\circ j(\gamma)
		 &= \Phi^L(X\setminus E, \w,0)\circ U_{1}^L\circ j(\gamma)\\
		&= \Phi^L(X\setminus E,\w,0)[\emptyset]\\
		&= [\emptyset].
	\end{align*}
	Hence
	\[[\emptyset] \in \im\bigg(U_E^L: RH^L(\pd E,\lambda_\std)\to RH^L(\pd E, \lambda_\std)\bigg).\]
	But the shortest Reeb orbit of $\pd E$ has length $2L$. So $RH^L(\pd E, 0)$ is isomorphic to $\Q$, generated by $[\emptyset]$, with a trivial $U_E$ action. We thus obtain a contradiction. 
	
	By permuting the ends, we see the RSFT spectrum is infinite for each boundary component $(Y_i,\lambda_i)$. If $(X_i,\w_i)$ is a weak filling of $(Y_i,\tilde{\lambda_i})$, and $\tilde{\lambda_i}$ is any contact form for $\ker(\lambda_i)$, then $c_k^\mathrm{RSFT}(X_i,\w_i) = c_k^\mathrm{RSFT}(Y_i,\tilde{\lambda_i})=\infty$,
	by \Cref{lem:rsftcap}. Lastly, from the proof of the disjoint union property in \Cref{lem:rsftcapprops},
	\[c_k^\mathrm{RSFT}(X,\w)= c_k^\mathrm{RSFT}(Y,\lambda) \geq c_k^\mathrm{RSFT}(Y_1,\lambda_1)=\infty. \]
\end{proof}

\begin{corollary}
	Let $\Sigma_g$ be a closed surface of genus $g\geq 2$. Then the RSFT capacities of a cotangent disk bundle $D^\ast\Sigma_g$ with its standard Liouville structure are infinite.
\end{corollary}
\begin{corollary}
	If $(X,\w)$ is a symplectic 4-manifold whose first RSFT capacity is finite, then $X$ contains no closed embedded Lagrangian surface of genus $g\geq 2$.
\end{corollary}

In higher dimensions, \Cref{prop:rsft} should already be non-trivial for the examples of Geiges and Massot--Niedekr\"uger--Wendl \cite{Gei6,MNW}. 

The following is well known \cite{McD} but we explain how to prove it within our framework. It will also provide a sample computation of the capacities, which recovers the Gromov non-squeezing theorem. Similar restrictions on co-fillings have also been obtained using RSFT by Moreno--Zhou \cite{MZ2}.
\begin{corollary}
	If a closed contact $2n-1$ manifold $(Y, \lambda)$ embeds as a contact-type hypersurface in $\R^{2n}$, then it is not strongly co-fillable.
\end{corollary}
\begin{proof}
	Suppose a hypersurface $(Y,\lambda)$ was strongly co-fillable. Necessarily, $(Y,\lambda)$ bounds a compact subdomain $(X,\w)$ of $\R^{2n}$. And $(X,\w)$ must be contained in some large ball $(B^{2n}(R), \w)$ inside $\R^{2n}$. The complement $B^{2n}(R) \setminus X$ defines a weakly exact cobordism either from $(S^{2n-1},\xi_\std)\sqcup (Y,\lambda)$ to $\varnothing$ or from $(S^{2n-1}, \xi_{\std})$ to $(Y,\lambda)$. In the first case, this is a strong co-filling of $(S^{2n-1},\xi_\std)$. In the second case, composing with the strong co-filling of $(Y,\lambda)$ gives a strong co-filling of $(S^{2n-1},\xi_\std)$. Hence it suffices to show $(S^{2n-1},\xi_\std)$ has no strong co-filling. 
	
	In light of \Cref{prop:rsft}, we just need to show the first RSFT capacity of some star-shaped domain is finite. For this we may consider an ellipsoid 
	\[E:= E(a_1,\ldots, a_n) = \left\{z \in \C^n \bigg| \sum_{i=1}^n \frac{\pi |z_i|^2}{a_i} \leq 1\right\},\]
	where $0< a_1 <a_2 < \cdots < a_n$ form a linearly independent set over $\Q$. Then $(\pd E, \lambda_\std)$ is a non-degenerate contact manifold and it has $n$ simple Reeb orbits $\gamma_1,\ldots, \gamma_n$ of periods $a_1,\ldots, a_n$ respectively. All the simple orbits and their covers have distinct odd Conley--Zehnder indices, so all Reeb orbits $\gamma$ are good and their corresponding variables $q_\gamma$ have even grading. 
	
	Let $\mathcal{A}$ be the ideal of non-constant polynomials inside the polynomial algebra generated by variables $q_{\gamma}$ for $\gamma$ a $k$-fold cover of $\gamma_i$ with $i\in \{1,\ldots,  n\}$ and $k$ any positive integer, and $\mathfrak{A} = \mathrm{Sym}(\mathcal{A})$. The RSFT differential is trivial for grading reasons and so $RH(\pd E, \lambda_\std)= \mathfrak{A}$. To show the first RSFT capacity is finite, it suffices to show $[\emptyset]$ is in the image of the $U$ map on $RH(\pd E, \lambda_\std)$. For grading and index reasons, the only homogeneous element which can have $[\emptyset]$ in its image is $[q_{\gamma_1}]$ and $U[q_{\gamma_1}]$ is always some multiple of $[\emptyset]$. We need to know it is a non-zero multiple, i.e. the algebraic count of holomorphic planes  asymptotic to $\gamma_1$ and satisfying a point constraint is non-zero.
	
	There is a similar augmentation  
	\[\epsilon_E \langle p \rangle : \mathrm{CHA}(\pd E, \lambda_\std) \to \Q \]
	defined on the contact homology algebra of $\pd E$ which counts holomorphic planes in the ellipsoid $E$ with a point constraint. For the same index and grading reasons, the contact homology algebra is formal (i.e. the differential is trivial) and the augmentation only counts holomorphic planes positively asymptotic to $\gamma_1$ with a point constraint. As explained by Siegel \cite[Rmk.\!\! 5.3]{Sie}, counting holomorphic planes in the ellipsoid $E$ is equivalent to counting planes in the symplectization $\R\times \pd E$. Hence, assuming a coherent system of virtual perturbations for all rational curve counts, $U[q_{\gamma_1}]$ is non-zero if and only if the augmentation $ \epsilon_E \langle p \rangle $ is non-trivial. 
	
	So to see $c_1^{\mathrm{RSFT}}(E,\w)$ is finite, we just need to know $\epsilon_E\langle p\rangle$ is non-trival. Consider an embedding of $E$ (appropriately scaled) into $\C P^n$ and let $W$ denote its complement. We can define a Cieliebak--Latschev element $\mathfrak{cl}_{W,H,1}\in\mathrm{CHA}(\pd E,\lambda_\std)$, as in \cite[\S 4]{Sie}, which counts point-constrained curves in $W$. More precisely, fix a cobordism-compatible almost-complex structure $J$ on the completion $\widehat{W}$ and a point $p\in \widehat{W}$. Then
	\[\mathfrak{cl}_{W,H, 1}: = \sum_{\Gamma} \frac{1}{\kappa_{\Gamma}}\#^\mathrm{vir}\left(\big\{ u\in \mathcal{M}^\mathrm{ind\, 2n-2}_{0,1} (\widehat{W}, J, H, \varnothing, \Gamma)| \mathrm{ev}_1(u) = p\big\}\right) \Gamma,\]
	where we sum over lists of orbits $\Gamma$ respecting a chosen ordering, $H\in H_2(\C P^n)$ denotes the class of a complex line, and the term $\#^\mathrm{vir}(\cdots )\in\Q$ is the orbifold count of genus-zero index $2n-2$ finite energy $J$-holomorphic curves in $\widehat{W}$ with no positive ends, negatively asymptotic to $\Gamma$, representing the class $H$, and with a marked point mapping to $p$. By standard arguments, this is a cycle and its homology class is independent of the choice of $p$ and $J$ (this is trivial in our case because the contact homology algebra is formal). 
	
	One can apply the above augmentation $\epsilon_E\langle p\rangle$ to this Cieliebak--Latschev element to obtain a rational number. By a standard gluing argument, we see that
	\[\epsilon_E\langle  p\rangle (\mathfrak{cl}_{W,H,1}) = \mathrm{GW}_{0,2} (\C P^n, H;\mathrm{PD}(\mathrm{pt}),\mathrm{PD}(\mathrm{pt})).\]
	The righthand side of the equality is the Gromov--Witten invariant of $\C P^n$ counting rational curves representing the homology class $H\in H_2(\C P^n)$ with two point constraints. Because every two points in $\C P^n$ define a unique complex line, and the standard integrable almost-complex structure satisfies the criteria for automatic transversality, this Gromov--Witten invariant equals one. Thus, the map $\epsilon_E\langle p\rangle$ cannot be zero. We deduce
	\[c_1^\mathrm{RSFT}(E(a_1,\ldots,a_n),\w_\std)=a_1<\infty.\]
	\end{proof}
	 

\end{document}